\documentclass[letterpaper,10pt]{article}

\usepackage[utf8]{inputenc}
\usepackage{lmodern}
\usepackage{csquotes}
\usepackage[english]{babel}

\usepackage{geometry}
\usepackage{microtype}
\usepackage{enumitem}

\usepackage{graphicx}
\usepackage{hyperref}
\usepackage{xcolor}
\definecolor{url}{HTML}{2A1B81}
\definecolor{cite}{HTML}{0C5724} 
\definecolor{link}{HTML}{800F0F} 
\definecolor{anchor}{HTML}{06245B} 
\hypersetup{colorlinks=true,
            urlcolor=,
            citecolor=cite,
            linkcolor=link,
            anchorcolor=anchor
}
\hypersetup{
    pdfauthor   = {Yannick Voglaire and Ping Xu},%
    pdftitle    = {Rozansky–Witten-type invariants from symplectic Lie pairs},%
    pdfsubject  = {We introduce symplectic structures on Lie pairs of (real or complex) algebroids as studied by Chen, Stiénon, and the second author, encompassing homogeneous symplectic spaces, symplectic manifolds with a g-action, and holomorphic symplectic manifolds. We show that to each such symplectic Lie pair are associated Rozansky-Witten-type invariants of three-manifolds and knots, given respectively by weight systems on trivalent and chord diagrams.},%
    pdfkeywords = {3-manifold invariants,knot invariants,weight systems,Rozansky-Witten,Lie algebroids,symplectic geometry,connections,Atiyah class},%
    pdfproducer = {PDFLaTeX}
}

\usepackage{amsmath,amsthm,amssymb}
\usepackage{breqn}
\usepackage[all]{xy}

\newtheorem{thm}{Theorem}[section]
\newtheorem{lem}[thm]{Lemma}

\newtheorem{defn}[thm]{Definition}
\newtheorem{prop}[thm]{Proposition}

\newcommand{\exend}{
    \hbox{}\nobreak\hfill\ensuremath{\lhd}
}

\theoremstyle{definition}
\newtheorem{ex}[thm]{Example}
\newtheorem{rem}[thm]{Remark}

\numberwithin{equation}{section}
\numberwithin{figure}{section}

\newcommand{\fR}{\mathbb{R}} 
\newcommand{\fC}{\mathbb{C}}
\newcommand{\fK}{\mathbb{K}}
\newcommand{\fN}{\mathbb{N}}

\DeclareMathOperator{\Span}{Span}
\DeclareMathOperator{\End}{End}
\DeclareMathOperator{\id}{id}
\DeclareMathOperator{\Ad}{Ad}
\DeclareMathOperator{\ad}{ad}
\DeclareMathOperator{\rk}{rk}
\DeclareMathOperator{\sign}{sign}

\DeclareMathOperator{\Adm}{Adm}
\DeclareMathOperator{\Tr}{Tr}

\newcommand{\g}{\mathfrak{g}}
\newcommand{\h}{\mathfrak{h}}
\newcommand{\mdp}{\Join} 
\renewcommand{\L}{\mathcal{L}} 

\newcommand{\so}{\textbf{s}}
\newcommand{\ta}{\textbf{t}}

\renewcommand{\i}{\iota}

\newcommand{\ol}[1]{\overline{#1}}

\let\oldemph\emph
\renewcommand{\emph}[1]{\textbf{#1}}

\newcommand{\trans}{\rtimes}

\DeclareFontFamily{U}{matha}{\hyphenchar\font45}
\DeclareFontShape{U}{matha}{m}{n}{
      <5> <6> <7> <8> <9> <10> gen * matha
      <10.95> matha10 <12> <14.4> <17.28> <20.74> <24.88> matha12
      }{}
\DeclareSymbolFont{matha}{U}{matha}{m}{n}
\DeclareFontSubstitution{U}{matha}{m}{n}
\DeclareMathSymbol{\abxcup}{\mathbin}{matha}{'131}

\title{Rozansky--Witten-type invariants from symplectic Lie pairs}
\author{Yannick Voglaire%
\thanks{Research partially supported by the Fonds National de la Recherche, Luxembourg, through the AFR Grant PDR 2012-1 (Project Reference 3966341).}
\\
Mathematics Research Unit, University of Luxembourg\\
\href{mailto:yannick.voglaire@gmail.com}{yannick.voglaire@gmail.com}
\and
Ping Xu%
\thanks{Research partially supported by the National Science Foundation Grant DMS-1101827.}
\\
Department of Mathematics, Penn State University\\
\href{mailto:ping@math.psu.edu}{ping@math.psu.edu}
}

\begin{document}

\maketitle

\begin{abstract}
  We introduce symplectic structures on ``Lie pairs'' of (real or complex) Lie algebroids as studied by Chen, Stiénon, and the second author in \cite{chen_atiyah_2012}, encompassing homogeneous symplectic spaces, symplectic manifolds with a $\g$-action, and holomorphic symplectic manifolds.
  We show that to each such symplectic Lie pair are associated Rozansky--Witten-type invariants of three-manifolds and knots, given respectively by weight systems on trivalent and chord diagrams.
\end{abstract}

\tableofcontents


\section{Introduction}
\label{sec:introduction}

In \cite{rozansky_hyper-kahler_1997}, Rozansky and Witten defined finite-type invariants of 3-manifolds associated to hyper-Kähler manifolds, by constructing weight systems on trivalent graphs.

Weight systems are linear functionals on the $\fC$-vector space generated by all trivalent diagrams, subject to the so-called AS and IHX relations (see \cite{bar-natan_vassiliev_1995,garoufalidis_finite_1998}).
They enable the construction of finite-type invariants of integral homology 3-spheres by precomposing them with a ``universal'' finite-type invariant such as the LMO invariant \cite{le_universal_1998}, associating to each integral homology 3-sphere a linear combination of trivalent diagrams.

Shortly after \cite{rozansky_hyper-kahler_1997}, Kontsevich \cite{kontsevich_rozanskywitten_1999} and Kapranov \cite{kapranov_rozanskywitten_1999} realized that the construction extended to a broader context and that, in particular, the hyper-Kähler metric was not required: a holomorphic symplectic manifold was sufficient.
See \cite{sawon_rozanskywitten_1999,kapustin_topological_2010,qiu_knot_2012} for further references about Rozansky--Witten invariants.
Kapranov \cite{kapranov_rozanskywitten_1999} showed moreover that the whole construction relied on a single cohomology class---the Atiyah class of the underlying complex manifold---measuring the obstruction to the existence of a holomorphic connection. 

Atiyah classes attracted much attention in the last decade (see for example \cite{ramadoss_big_2008,markarian_atiyah_2009,roberts_rozanskywitten_2010}).
Recently, Chen, Stiénon, and the second author \cite{chen_atiyah_2012} defined a notion of Atiyah class for a pair of (real or complex) Lie algebroids $A\subset L$ over a manifold $M$ (a \emph{Lie pair}, in short) and for an $A$-module $E\to M$, 
generalizing the original case of complex manifolds \cite{atiyah_complex_1957}, 
the Atiyah--Molino class for connections transverse to a foliation \cite{molino_classe_1971,vaisman_sur_1973,kamber_characteristic_1974}, 
and the obstruction to the existence of invariant connections on homogeneous spaces \cite{wang_invariant_1958,vinberg_invariant_1960,hai_conditions_1964}.

The present paper is a natural follow-up of \cite{chen_atiyah_2012}. 
We introduce symplectic structures on Lie pairs and show that, as in Kapranov's work on holomorphic symplectic manifolds, a symplectic Lie pair induces a weight system on trivalent diagrams.
Given a Lie pair $A\subset L$ with a symplectic structure $\omega\in\Gamma(\Lambda^2(L/A)^*)$, the constructed weight system takes values in the Lie algebroid cohomology of $A$ with trivial coefficients, and sends trivalent diagrams with $2k$ vertices to $H^{2k}(A)$.
Unlike in the holomorphic symplectic case, where $L/A$ and $A$ are the holomorphic and antiholomorphic tangent bundles, it is in general not possible to extract numerical invariants in a canonical way from the weights in $H^{2k}(A)$.

Recently, the Rozansky--Witten theory was studied from the viewpoint of
categorified algebraic geometry by Kapustin and Rozansky \cite{kapustin_three-dimensional_2010}.
An AKSZ-type construction of a sigma model underlying a symplectic Lie pair and its connections
with derived geometry will be studied elsewhere.

The paper is organized as follows.
In Section~\ref{sec:preliminaries}, we recall some basic facts and fix the notation.
In Section~\ref{sec:symplectic-lie-pairs}, we introduce symplectic structures on Lie pairs, provide some examples, and explain the meaning of such a structure on the groupoid level.
In Section~\ref{sec:atiyah-classes}, we adapt the classical notions and results about symplectic connections to the context of symplectic Lie pairs, and use these to build a totally symmetric Atiyah cocycle.
In Section~\ref{sec:weight-systems}, we construct the announced weight systems, essentially following Rozansky and Witten \cite{rozansky_hyper-kahler_1997}, Kapranov \cite{kapranov_rozanskywitten_1999}, and Sawon \cite{sawon_rozanskywitten_1999}.


\subsection*{Acknowledgments}
We would like to express our gratitude to the following institutions for their hospitality while we were working on this project: Université Paris Diderot-Paris~7 (Xu) and Penn State University (Voglaire).
We would also like to thank Justin Sawon and Mathieu Sti\'enon for fruitful discussions, as well as the anonymous reviewer for useful suggestions.


\section{Preliminaries}
\label{sec:preliminaries}
We fix here some terminology and notation that will be used throughout the paper.

A \emph{Lie algebroid} over $\fK$ (with $\fK=\fR$ or $\fC$) is a $\fK$-vector bundle $L\to M$ over a smooth manifold $M$, together with a bundle map $\rho:L\to TM\otimes_\fR\fK$ called the \emph{anchor}, and a bracket $[\cdot,\cdot]$ on the sections of $L$, such that
\[
  [l_1,fl_2]=f[l_1,l_2]+\rho(l_1)(f)l_2
\]
for all $l_1,l_2\in\Gamma(L)$ and $f\in C^\infty(M)\otimes_\fR\fK$.
In that case, $\rho$ seen as a map of sections is a morphism of Lie algebras.

An $L$-connection on a $\fK$-vector bundle $E\to M$ is a $\fK$-bilinear map
\[
  \nabla:\Gamma(L)\times\Gamma(E)\to\Gamma(E)
\]
such that
\[
  \nabla_{fl}e = f\nabla_le \qquad \text{and} \qquad \nabla_l(fe)=f\nabla_le+\rho(l)(f)e
\]
for all $l\in\Gamma(L)$, $e\in\Gamma(E)$ and $f\in C^\infty(M)\otimes_\fR\fK$.
The \emph{curvature} of $\nabla$ is the bundle map $R:L\otimes L\to \End(E)$ defined on sections by
\[
  R(l_1,l_2) = \nabla_{l_1}\nabla_{l_2} - \nabla_{l_2}\nabla_{l_1} - \nabla_{[l_1,l_2]}
\]
for all $l_1,l_2\in \Gamma(L)$. 
An $L$-connection is \emph{flat} if its curvature identically vanishes.
A vector bundle $E\to M$ together with a flat $L$-connection is called an \emph{$L$-module}, and the connection itself is called a representation of $L$. 

A \emph{Lie pair} is a pair $(L,A)$ of Lie algebroids over the same manifold, with $A$ a Lie subalgebroid of $L$.
We denote $l\mapsto \ol l$ the canonical projection $L\to L/A$.
Although it is not an $L$-module in general, the quotient $L/A$ is automatically an $A$-module with connection $\nabla^A$ defined by $\nabla^A_a\ol l=\ol{[a,l]}$ for all $a\in\Gamma(A)$ and $l\in\Gamma(L)$.

Let $(E,\nabla^A)$ be an $A$-module,
and let $\Omega^k(A,E)=\Gamma(\Lambda^{k}A^*\otimes E)$ denote the $E$-valued $k$-forms on $A$, $k\geq 0$.
There is a differential $\partial^A$ on the complex $\Omega^\bullet(A,E)$ which extends the flat connection $\nabla^A$ seen as a map $\Omega^0(A,E)\to \Omega^1(A,E)$.
It is defined by
\begin{dmath}
\label{eq:definition-of-partial-A}
  (\partial^A \eta)(a_0,\dots,a_k)
  = \sum_{i=0}^k (-1)^i \nabla^A_{a_i}(\eta(a_0,\dots,\hat a_i,\dots,a_k))
  + \sum_{i<j} (-1)^{i+j} \eta([a_i,a_j],a_0,\dots,\hat a_i,\dots, \hat a_j,\dots,a_k) ,
\end{dmath}
for all $\eta\in\Omega^k(A,E)$ and $a_0,\dots,a_k\in\Gamma(A)$.
It gives rise to the Lie algebroid cohomology $H^\bullet(A,E)$ of $A$ with coefficients in $E$.
When the $A$-module $E$ is the trivial line bundle $M\times \fK$ with action $\nabla^A_af=\rho(a)(f)$, the differential $\partial^A$ is written $d_A:\Omega^\bullet(A)\to\Omega^{\bullet+1}(A)$,
and the corresponding Lie algebroid cohomology is denoted by $H^\bullet(A)$.

Note also that if $(E,\nabla^A)$ is an $A$-module, then $E^{\otimes k}\otimes (E^*)^{\otimes l}$ is naturally an $A$-module, for all $k,l\geq 0$.
The associated flat connection, still denoted by the same symbol $\nabla^A$, is defined by
\begin{equation}
\label{eq:extension-of-connection-to-tensor-powers}
  \begin{aligned}
    \nabla^A_af 
    &= \rho(a)(f) ,
    \\
    \left \langle \nabla^A_a\epsilon, e \right \rangle 
    &= \rho(a) ( \left \langle \epsilon, e \right \rangle )
       - \left \langle \epsilon, \nabla^A_a e \right \rangle ,
  \end{aligned}
\end{equation}
for all $f\in C^\infty(M)\otimes_\fR\fK$, $\epsilon\in\Gamma(E^*)$, and $e\in\Gamma(E)$,
and then extended by derivations.

Given an $A$-module $(E,\nabla^A)$ and an $L$-connection $\nabla$ on $E$ \emph{extending the $A$-action} in the sense that $\nabla_ae=\nabla^A_ae$ for all $a\in\Gamma(A)$ and $e\in\Gamma(E)$, it follows \cite[Theorem 16]{chen_atiyah_2012} that the 1-form
\[
  R^\nabla_E\in\Omega^1(A,(L/A)^*\otimes \End(E))
\]
defined by
\[
  R^\nabla_E(a)(\ol l,e)=\left( \nabla_a\nabla_l - \nabla_l\nabla_a - \nabla_{[a,l]} \right) e ,
\]
for all $a\in\Gamma(A)$, $l\in\Gamma(L)$, and $e\in\Gamma(E)$, is $\partial^A$-closed.
Moreover, the cohomology class 
\[
  \alpha_E\in H^1(A,(L/A)^*\otimes \End(E))
\]
that it induces is independent of the chosen connection $\nabla$ extending the $A$-action, and is called the \emph{Atiyah class of $E$}. 

The construction of $\alpha_E$ applies in particular to the $A$-module $L/A$ to give the \emph{Atiyah class $\alpha_{L/A}$ of the Lie pair $(L,A)$}.

The Atiyah class measures the obstruction to finding an \emph{$A$-compatible} $L$-connection on $E$ in the sense of \cite{chen_atiyah_2012}, i.e., an $L$-connection on $E$ extending the $A$-action and such that $R^\nabla_E=0$.
Geometrically, its significance is described by Proposition~\ref{prop:atiyah-zero-equiv-invariant-connection} below. 
To state it, let us first introduce some notation.

Assume that in a Lie pair $(L,A)$ over $\fR$, $L$ integrates to a Lie groupoid $\Gamma_L$ and $A$ to a wide closed Lie subgroupoid $\Gamma_A$. (A subgroupoid of a Lie groupoid is said to be \emph{wide} if it is on the same base manifold, and \emph{closed} if it is a closed embedded submanifold.)
The corresponding ``homogeneous space'' $\Gamma_L/\Gamma_A$ was first considered in \cite{liu_dirac_1998}.
By \cite[Sec.\ 3]{moerdijk_integrability_2006}, $\Gamma_L/\Gamma_A$ is a smooth (Hausdorff) manifold such that $\Gamma_L\to \Gamma_L/\Gamma_A$ is a submersion.
The source map $\so:\Gamma_L\to M$ of $\Gamma_L$ induces a surjective submersion $J:\Gamma_L/\Gamma_A\to M$,
and the left translations in $\Gamma_L$ naturally induce a left action on $\Gamma_L/\Gamma_A$.
By a \emph{fibrewise affine connection} on $\Gamma_L/\Gamma_A$, we mean a smooth map
\[
  \nabla:\Gamma(\ker T(J))\otimes \Gamma(\ker T(J))\to \Gamma(\ker T(J))
\]
such that
\[
  \nabla_{fX}Y=f\nabla_XY \quad \text{and} \quad \nabla_X(fY)=X(f)Y+f\nabla_XY
\] 
for all $X,Y\in\Gamma(\ker T(J))$ and $f\in C^\infty(\Gamma_L/\Gamma_A)$.
This is just a smooth collection of affine connections on the $J$-fibers.
The connection is said to be \emph{invariant} if it is invariant under the action of $\Gamma_L$ on $\Gamma_L/\Gamma_A$.

\begin{prop}[{\cite{laurent-gengoux_invariant_2014}}]
\label{prop:atiyah-zero-equiv-invariant-connection}
  Let $\Gamma_L$ be a Lie groupoid and $\Gamma_A$ be a wide closed $\so$-connected Lie subgroupoid, with Lie algebroids $L$ and $A$, respectively.
  Then there is a one-to-one correspondence between $A$-compatible $L$-connections on $L/A$ and fibrewise affine connections on $\Gamma_L/\Gamma_A$.
  In particular, the Atiyah class $\alpha_{L/A}$ of the Lie pair $(L,A)$ vanishes if and only if there exists an invariant fibrewise affine connection on $\Gamma_L/\Gamma_A$.
\end{prop}

See also \cite{laurent-gengoux_kapranov_2014} for a similar result in the sense of formal neighborhoods.

In the case of pairs of Lie algebras $(\g,\h)$, the relation between bilinear maps 
\[
  \nabla:\g\times \g/\h\to \g/\h
\] 
such that 
\[
  \nabla_z=\ad_z \quad \text{and} \quad \nabla_z\nabla_x-\nabla_x\nabla_z-\nabla_{[z,x]}=0
\]
for all $x\in\g,z\in\h$, and invariant connections on a corresponding homogeneous space $G/H$ is well-known since the 60's, see \cite{wang_invariant_1958,vinberg_invariant_1960,hai_conditions_1964}, and more recently \cite{bordemann_atiyah_2012}.

As in the Lie algebra case, for Lie algebroids, a statement similar to Proposition~\ref{prop:atiyah-zero-equiv-invariant-connection} holds, concerning the Atiyah class of any $A$-module $E$ in relation to the existence of invariant fibrewise connections on a natural vector bundle over $\Gamma_L/\Gamma_A$ associated to $E$ \cite{laurent-gengoux_invariant_2014}.

For a permutation $\sigma$ of $\{1,\dots,n\}$, we define a linear transformation $\tau_\sigma$ permuting the $n$ components of a tensor product according to $\sigma$:
\begin{align}
\label{eq:definition-tau}
  \tau_\sigma : V_1\otimes \cdots \otimes V_n  &    \to  V_{\sigma(1)} \otimes \cdots \otimes V_{\sigma(n)} \\
                v_1\otimes \cdots \otimes v_n  &\mapsto  v_{\sigma(1)} \otimes \cdots \otimes v_{\sigma(n)}  \nonumber
\end{align}
where the vector bundles $V_i$ will be clear from the context (see Subsections~\ref{ssec:symmetries-of-the-atiyah-cocycle} and \ref{ssec:weight-systems}).
Explicit permutations will be written in canonical form.
For example, for a set of five elements, the notations $(143)(25)$ and $(12)$ respectively represent the permutations
\begin{align*}
  &
  \begin{pmatrix}
    1 & 2 & 3 & 4 & 5 \\
    4 & 5 & 1 & 3 & 2
  \end{pmatrix} ,
  &
  &
  \begin{pmatrix}
    1 & 2 & 3 & 4 & 5 \\
    2 & 1 & 3 & 4 & 5
  \end{pmatrix} .
\end{align*}


\section{Symplectic Lie pairs}
\label{sec:symplectic-lie-pairs}

\begin{defn}
\label{def:presymplectic}
  A \emph{presymplectic structure} on a Lie algebroid $L$ is a $d_L$-closed 2-form $\Omega\in\Gamma(\Lambda^2 L^*)$ with constant rank.
  The pair $(L,\Omega)$ is then called a \emph{presymplectic Lie algebroid}.
\end{defn}

\begin{rem}
  Any Lie algebroid $L$ gives rise to a Lie bialgebroid \cite{mackenzie_lie_1994} $(L,L^*)$ by endowing $L^*$ with the trivial Lie algebroid structure.
  This defines a Courant algebroid structure \cite{liu_manin_1997} on $L\oplus L^*$.
  Theorem 6.1 (together with Remark (1) following it) in the latter reference shows that a 2-form $\Omega\in\Gamma(\Lambda^2L^*)$ on a Lie algebroid $L$ is presymplectic in the sense of Definition~\ref{def:presymplectic} if and only if
  \begin{equation*}
    D=\{(l,\xi) \mid \xi=\Omega^\flat(l)\}
  \end{equation*}
  is a Dirac structure on $L\oplus L^*$, where $\Omega^\flat:L\to L^*$ is the induced bundle map corresponding to $\Omega$.
\end{rem}

\begin{defn}
  A \emph{symplectic structure} on a Lie pair $(L,A)$ is a non-degenerate 2-form $\omega\in \Gamma(\Lambda^2(L/A)^*)$ such that $d_L (p^*\omega)=0$, where $p:L\to L/A$ is the canonical projection.
  The triple $(L,A,\omega)$ is then called a \emph{symplectic Lie pair}.
\end{defn}

\begin{prop}
  Let $L$ be a Lie algebroid.
  There is a bijection between presymplectic structures on $L$ and symplectic Lie pairs $(L,A,\omega)$.
\end{prop}

\begin{proof}
  The kernel $\ker \Omega$ of a 2-form $\Omega\in\Gamma(\Lambda^2 L^*)$ is defined by $\ker \Omega= \{ v \in L \mid \iota_v\Omega = 0 \}$.
  Obviously, if $\omega$ is a symplectic structure on a Lie pair $(L,A)$, then $\Omega=p^*\omega$ is a presymplectic structure on $L$ with kernel equal to $A$.

  Conversely, let $\Omega\in\Gamma(\Lambda^2 L^*)$ be a presymplectic structure on $L$, and let $A=\ker \Omega$.
  Since $\Omega$ has constant rank, its kernel $A$ has constant rank as well.
  And the closedness of $\Omega$ yields
  \begin{align*}
    0
    &=
    \left(d_L\Omega\right)(a,a',l)
    \\
    &=
    \rho(a)\left(\Omega(a',l)\right) - \rho(a') \left(\Omega(a,l)\right) + \rho(l) \left(\Omega(a,a')\right) \\
    &\mathbin{\phantom{=}}
    {} - \Omega([a,a'],l) + \Omega([a,l],a') - \Omega([a',l],a) \\
    &=
    -\Omega([a,a'],l)
  \end{align*}
  for all $a,a'\in\Gamma(A)$, $l\in\Gamma(L)$, which implies $[\Gamma(A),\Gamma(A)]\subset \Gamma(A)$.
  Hence, $A$ is a Lie subalgebroid of $L$.
  Since $A=\ker \Omega$, the form $\Omega$ descends as a 2-form $\omega$ on $L/A$ such that $\Omega=p^*\omega$.
  We have $d_L(p^*\omega)=d_L\Omega=0$, and therefore $(L,A,\omega)$ is a symplectic Lie pair.

  The two constructions are obviously inverse of each other.
\end{proof}

The most obvious examples of presymplectic structures are of course the regular presymplectic manifolds.
\begin{prop}
  The regular presymplectic forms on a manifold $M$ are in bijection with the presymplectic structures on the Lie algebroid $TM$.
\end{prop}

In the next three subsections, we describe a few other examples.


\subsection{Holomorphic symplectic manifolds}

Let $(X,\omega)$ be a holomorphic symplectic manifold, i.e., let $X$ be a complex manifold and $\omega\in\Omega^{2,0}(X)$ be a closed non-degenerate holomorphic $(2,0)$-form.
Note that for a $(2,0)$-form, closed actually implies holomorphic, since $d=\partial+\ol{\partial}$ and $\partial\omega\in\Omega^{3,0}(X)$, $\ol{\partial}\omega\in\Omega^{2,1}(X)$.

Consider the complex Lie algebroids $L=TX\otimes \fC \cong T^{1,0}X\oplus T^{0,1}X$ and $A=T^{0,1}X$.
Then $\omega$ is a $2$-form on $L$ with kernel $A$, and $d\omega=0$ is equivalent to $d_L\omega=0$.
Hence $(L,\omega)$ is a presymplectic Lie algebroid. Equivalently, $(L,A,\omega')$ is a symplectic Lie pair, where $\omega'\in\Gamma(\Lambda^2(L/A)^*)$ is $\omega$ seen as a 2-form on $L/A\cong T^{1,0}X$.

\begin{prop}
  Let $X$ be a complex manifold. 
  There is a bijection between the holomorphic symplectic structures on $X$ and the symplectic structures on the Lie pair $(TX\otimes\fC,T^{0,1}X)$.
\end{prop}

The Lie pair $(L,A)$ associated to a complex manifold $X$ as above has a structure of \emph{matched pair} of Lie algebroids \cite{mokri_matched_1997}: $L$ decomposes, as a vector bundle, as the direct sum $L=B\oplus A$ of two Lie subalgebroids $A$ and $B$.
In that case, $L/A\cong B$ is again a Lie algebroid, endowed with a representation of $A$ (i.e., a flat $A$-connection $\nabla^{A}$ on $B$), and a similar statement holds for $L/B$. Matched pairs will be denoted by $L=B\mdp A$.

A symplectic structure on the Lie pair $(B\mdp A, A)$ induced from a matched pair is a 2-form on $B$ which is $d_B$-closed and $A$-invariant: $\omega\in\Gamma(\Lambda^2B^*)$ such that $d_B\omega=0$ and $\partial^{A}\omega=0$.
Here, $\partial^A$ is the Chevalley--Eilenberg differential on $\Omega^\bullet(A,\Lambda^\bullet B^*)$ induced from $\nabla^A$ (see Eq.~\eqref{eq:definition-of-partial-A} and \eqref{eq:extension-of-connection-to-tensor-powers}).

In the case of the Lie pair $(L,A)$ associated to a complex manifold $X$, we have $A=T^{0,1}X$ and $B=T^{1,0}X$.
Hence, 
\[
  \Gamma(\Lambda^\bullet B^*)=\Omega^{\bullet,0}(X) \quad\text{and}\quad \Omega^\bullet(A,\Lambda^\bullet B^*)=\Omega^{\bullet,\bullet}(X), 
\] 
and the operators $d_B$ and $\partial^A$ are respectively 
\begin{align*}
  d_B &=\partial:\Omega^{\bullet,0}(X)\to\Omega^{\bullet+1,0}(X), \\
  \partial^A &=\ol\partial:\Omega^{\bullet,\bullet}(X) \to \Omega^{\bullet,\bullet+1}(X).
\end{align*}


\subsection{(Pre)symplectic manifolds with a \texorpdfstring{$\g$}{g}-action}

Given a Lie algebra $\g$, a $\g$-action on a manifold $M$ gives rise \cite[Example 5.5]{mokri_matched_1997} to a matched pair of Lie algebroids $TM\mdp(M\trans\g)$ where $M\trans\g$ is the transformation Lie algebroid of the $\g$-action.

Recall that, for a $\g$-action $\phi:\g\to\Gamma(TM)$, the transformation Lie algebroid $M\trans\g$ is, as a vector bundle, the trivial bundle $M\times \g$ over $M$.
Its anchor is $\rho(x,X)=\phi(X)_x$.
Its Lie bracket is defined on constant sections as the pointwise Lie bracket, and then extended to all sections by linearity and by the Leibniz rule.

The Lie bracket on the matched pair $TM\mdp (M\trans\g)$ is defined by
\[
  [v_1+X_1,v_2+X_2] = ([v_1,v_2] + [\phi(X_1),v_2] - [\phi(X_2),v_1]) + [X_1,X_2],
\]
for all \emph{constant} sections $X_1,X_2\in\Gamma(M\trans\g)$ and all vector fields $v_1,v_2\in\Gamma(TM)$.

According to the previous subsection, a symplectic structure on the Lie pair $(TM\mdp(M\trans\g),M\trans\g)$ is a Lie algebroid closed 2-form on $TM$ which is $(M\trans\g)$-invariant, i.e., a $\g$-invariant symplectic form on $M$:

\begin{prop}
  Let $M$ be a manifold with a $\g$-action. 
  There is a bijection between the $\g$-invariant symplectic structures on $M$ and the symplectic structures on the Lie pair $(TM\mdp(M\trans\g),M\trans\g)$.
\end{prop}

This result can be generalized to Lie pairs with a $\g$-action, as we explain below.

Assume that a Lie algebra $\g$ acts on a Lie algebroid $L$ by derivations, in the sense that there exists a Lie algebra homomorphism $\phi:\g\to\End(\Gamma(L))$ such that 
\begin{align}
\label{eq:action-by-derivations}
  \phi(X)([l_1,l_2])&=[\phi(X)l_1,l_2] + [l_1, \phi(X) l_2]
\end{align}
for all $l_1,l_2\in\Gamma(L)$ and $X\in\g$.
This is equivalent to asking for a representation $\nabla$ of $M\trans \g$ on $L$ 
acting by derivations, in the sense that for all \emph{constant} sections $X\in\Gamma(M\trans\g)$ (with corresponding element $\overline X\in\g$),
the map $\phi(\overline X)=\nabla_X$ satisfies Eq.~\eqref{eq:action-by-derivations}.

Given such an action, one can form a
matched pair $L'=L\mdp (M\trans\g)$ as follows.
The Lie bracket is defined for constant sections $X_1,X_2\in\Gamma(M\trans\g)$ and for any $l_1,l_2\in\Gamma(L)$ by
\[
  [l_1+X_1,l_2+X_2] = ([l_1,l_2] + \nabla_{X_1}l_2 - \nabla_{X_2}l_1) + [X_1,X_2],
\]
and extended by linearity and the Leibniz rule.
The anchor is defined as the sum of the anchors of $L$ and $M\trans \g$.
This is the most general matched pair $L\mdp (M\trans\g)$ with trivial (i.e., vanishing on constant sections) representation of $L$ on $M\trans\g$, or in other words, the most general semi-direct product of $L$ and $M\trans\g$.

Assume now that $(L,A)$ is a Lie pair, and that $\nabla$ is an action of $M\trans \g$ on $L$ by derivations,
which preserves $A$.
By this, we mean that $\nabla_X\Gamma(A)\subset\Gamma(A)$ for all $X\in\Gamma(M\trans\g)$.
Then we have a new Lie pair $(L'=L\mdp(M\trans\g),A'=A\mdp(M\trans\g))$, and the following result holds.

\begin{prop}
  Let $(L,A)$ be a Lie pair.
  There is a bijection between the $\g$-invariant symplectic structures on $(L,A)$ and the symplectic structures on the Lie pair $(L\mdp(M\trans\g),A\mdp(M\trans\g))$.
\end{prop}


\subsection{Homogeneous symplectic spaces}

\begin{prop}
\label{prop:homogeneous-symplectic-spaces}
  Let $G/H$ be a homogeneous space with $H$ connected.
  Let $\g$ and $\h$ be the Lie algebras of $G$ and $H$, respectively.
  The $G$-invariant symplectic structures $\omega$ on $G/H$ are in bijection with the symplectic structures $\Omega$ on the Lie pair $(\g,\h)$.
  In the correspondence, $\Omega$ is the value of $\omega$ at the base point $eH$.
\end{prop}

We describe now an important subclass of homogeneous symplectic spaces.

In \cite{pikulin_invariant_2001}, Pikulin and Tevelev characterise the (covers of) adjoint orbits of semisimple Lie groups which admit invariant linear connections.
\begin{thm}[{\cite[Theorem 1]{pikulin_invariant_2001}}]
\label{thm:pikulin-tevelev}
Let $G$ be a connected semisimple Lie group over the field $\mathbb{K}=\fR$ or $\fC$, 
and let $X$ cover the adjoint orbit $\Ad(G)x$, $x\in\g$.
Let $x = x_{\mathrm s} + x_{\mathrm n}$ be the Jordan decomposition in $\g$, $\mathfrak{z}(x_{\mathrm s}) = \mathfrak{z}\oplus \g_1 \oplus \cdots \oplus \g_m$ the decomposition of the centralizer of $x_{\mathrm s}$ into a sum of the center $\mathfrak z$ and of simple ideals $\g_k$, and denote $x_{\mathrm n} = 	x_{\mathrm n}^1 + \cdots + x_{\mathrm n}^m$, where $x_{\mathrm n}^k\in\g_k$, $k = 1 , \dots , m$. 
Then $X$ admits an invariant linear connection if and only if, for each $k = 1 , \dots , m$, the following implication is true:
\[
  x_{\mathrm n}^k \neq 0 \Rightarrow \g_k \simeq \mathfrak{sp}(2n_k, \mathbb{K}), 
  \text{ $n_k \in \mathbb{N}$ and $x_{\mathrm n}^k$ is a highest root vector in $\g_k^\fC$.}
\]
\end{thm}

The adjoint orbit $\Ad(G)x$ of a nilpotent element $x$ (i.e., with Jordan decomposition $x=x_{\mathrm n}$) is called a \emph{nilpotent orbit}
(see \cite{collingwood_nilpotent_1993,jantzen_nilpotent_2004}).
Every nilpotent orbit of a semisimple Lie group, being equivariantly diffeomorphic to a coadjoint orbit, is a homogeneous symplectic space.
Hence, it defines a symplectic Lie pair as in Proposition~\ref{prop:homogeneous-symplectic-spaces}.
Moreover, recall from Proposition~\ref{prop:atiyah-zero-equiv-invariant-connection} that the Atiyah class of the Lie pair associated to a homogeneous space vanishes if and only if the latter space admits invariant affine connections.

Hence, Theorem~\ref{thm:pikulin-tevelev} above provides us with a wealth of examples of symplectic Lie pairs with non-vanishing Atiyah class.

For example, among the nilpotent orbits of real simple Lie groups, only the orbits of highest root vectors of the symplectic groups $Sp(2n,\fR)$ admit invariant connections, and thus have vanishing Atiyah class.


\subsection{Groupoid picture}
\label{ssec:groupoid-picture}

In the case of Lie algebroids over $\fR$, symplectic structures on Lie pairs correspond to fibrewise 
symplectic forms on the corresponding homogeneous spaces of Lie groupoids, if the latter exist.

More precisely, assume that in a Lie pair $(L,A)$, $L$ integrates to a Lie groupoid $\Gamma_L$ and $A$ to a closed Lie subgroupoid $\Gamma_A$.
Recall from Section~\ref{sec:preliminaries} that, by \cite[Sec.\ 3]{moerdijk_integrability_2006}, $\Gamma_L/\Gamma_A$ is then a smooth (Hausdorff) manifold such that $\Gamma_L\to\Gamma_L/\Gamma_A$ is a submersion.
The quotient $\Gamma_L/\Gamma_A$ is fibered over $M$ by a surjective submersion $J:\Gamma_L/\Gamma_A\to M$ induced by the source map $\so:\Gamma_L\to M$, and there is a natural $\Gamma_L$-action on $\Gamma_L/\Gamma_A$ induced from the left-translations in $\Gamma_L$. 
We call a \emph{fibrewise differential $k$-form} on $\Gamma_L/\Gamma_A$ a smooth section of the bundle $\Lambda^k ((\ker T(J))^*)$ over $\Gamma_L/\Gamma_A$, i.e., a smooth collection of de Rham differential $k$-forms on the $J$-fibers.

\begin{thm}
\label{thm:bijection-sympl-str-LA-and-fibrewise-sympl-str-GLGA}
  Let $\Gamma_L$ be a Lie groupoid and $\Gamma_A$ a wide closed $\so$-connected Lie subgroupoid, with Lie algebroids $L$ and $A$, respectively.
  There is a bijection between symplectic structures on $(L,A)$ and $\Gamma_L$-invariant fibrewise symplectic structures on $\Gamma_L/\Gamma_A$.
\end{thm}

We first prove a lemma which will readily imply the theorem.

Denote by $\so,\ta:\Gamma_L\to M$ the source and target maps of $\Gamma_L$, and take the convention that $g_1,g_2\in\Gamma_L$ are composable as $g_1 \cdot g_2$ if $\ta(g_1)=\so(g_2)$. We consider bisections of $\Gamma_L$ as sections $h$ of $\so$ such that $\ta\circ h$ is a diffeomorphism of $M$, and denote the composition of two bisections, or of a bisection and a groupoid element, by $\star$.
For a bisection $h$ of $\Gamma_L$, denote by 
\[
  I_h:\Gamma_L\to\Gamma_L:g\mapsto h\star g \star h^{-1} = h{\left((\ta\circ h)^{-1}(\so(g))\right)} \cdot g \cdot {\left(h{\left((\ta\circ h)^{-1}(\ta(g))\right)}\right)^{-1}}
\] 
the conjugation by $h$. 
It preserves the unit submanifold and sends $\so$-fibers to $\so$-fibers, so its differential at the unit manifold induces an adjoint action $\Ad_h:L \to L$ over $(\ta\circ h)^{-1}:M\to M$.

Recall that the \emph{kernel} of a $k$-form $\alpha$ on a vector bundle $E$ is the set $\ker\alpha=\{e\in E\mid \i_e\alpha = 0\}$, where $\i_e$ is the contraction by $e$.

\begin{lem}
\label{lem:forms_on_L-forms_on_GammaLoverGammaA}
  Let $\Gamma_L$ be a Lie groupoid and $\Gamma_A$ a wide closed Lie subgroupoid, with Lie algebroids $L$ and $A$, respectively.
  Let $\alpha$ be a $k$-form on $L$. 
  Then the left-invariant $k$-form on $\Gamma_L$ corresponding to $\alpha$ descends to a fibrewise $k$-form $\ol\alpha$ on $\Gamma_L/\Gamma_A$ if and only if 
  \begin{enumerate}[label=(\arabic*)]
    \item \label{lem:forms_on_L-forms_on_GammaLoverGammaA-1} 
      $A\subset \ker \alpha$, and
    \item \label{lem:forms_on_L-forms_on_GammaLoverGammaA-2} 
      $\Ad_h^*\alpha=\alpha$ for all bisections $h$ of $\Gamma_A$.
  \end{enumerate}
  If $\Gamma_A$ is $\so$-connected, then the conditions boil down to
  \begin{enumerate}[label=(\arabic*')]
    \item \label{lem:forms_on_L-forms_on_GammaLoverGammaA-1prime}
      $A\subset \ker \alpha \cap \ker d_L\alpha$.
  \end{enumerate}
  When conditions \ref{lem:forms_on_L-forms_on_GammaLoverGammaA-1} and \ref{lem:forms_on_L-forms_on_GammaLoverGammaA-2} are satisfied, then $\alpha$ is $d_L$-closed if and only if $\ol\alpha$ is closed.
\end{lem}

\begin{proof}
  Denote $\pi:\Gamma_L\to\Gamma_L/\Gamma_A$ the projection.
  The left-invariant $k$-form $\tilde\alpha$ on $\Gamma_L$ corresponding to $\alpha$ 
  \cite[pp.\ 166--167]{weinstein_extensions_1991} will descend to a fibrewise $k$-form on 
  $\Gamma_L/\Gamma_A$ if and only if 
  \begin{enumerate}[label=(\alph*)]
    \item \label{item:form-descends-a} 
      at each point $g\in\Gamma_L$, $\tilde\alpha_g$ descends to $T_{\pi(g)}(\Gamma_L/\Gamma_A)$, and
    \item \label{item:form-descends-b}
      for two different points $g$ and $g'$ in the same $\pi$-fiber, 
      $\tilde\alpha_g$ and $\tilde\alpha_{g'}$ descend to the same form.
  \end{enumerate}
  
  Condition~\ref{item:form-descends-a} is equivalent to $\ker {T(\pi)} \subset \ker \tilde\alpha$.
  Since $\ker T_g(\pi) = T_{\ta(g)}(L_g)(A_{\ta(g)})$ and $\ker \tilde\alpha_g = T_{\ta(g)}(L_g) (\ker \alpha_{\ta(g)})$ for all $g\in\Gamma_L$, the latter is equivalent to $A\subset\ker\alpha$, which is condition \ref{lem:forms_on_L-forms_on_GammaLoverGammaA-1} of the Lemma.

  When $A\subset\ker\alpha$, condition~\ref{item:form-descends-b} is equivalent to $R_h^*\tilde\alpha = \tilde\alpha$ for all bisections $h$ of $\Gamma_A$. 
  Using the left-invariance of $\tilde\alpha$, the condition becomes $\Ad_h^*\alpha=\alpha$.
  This proves that $\tilde\alpha$ descends to $\Gamma_L/\Gamma_A$ if and only if \ref{lem:forms_on_L-forms_on_GammaLoverGammaA-1} and \ref{lem:forms_on_L-forms_on_GammaLoverGammaA-2} hold.

  Now, the right-translation $R_{\exp(tZ)}$ is the flow of the left-invariant vector field $\tilde Z$ corresponding to $Z\in \Gamma(A)$.
  Hence, if $R_h^*\tilde\alpha=\tilde\alpha$ for all bisections $h$ of $\Gamma_A$, then $\L_{\tilde Z}\tilde\alpha=0$ and thus $\L_Z\alpha=0$ for all $Z\in\Gamma(A)$. ($\L$ denotes the Lie derivative.)
  Using the Cartan formula, we get $i_Zd_L\alpha = \L_Z\alpha - d_Li_Z\alpha = 0$ for all $Z\in \Gamma(A)$, that is, $A\subset \ker d_L\alpha$.
  When $\Gamma_A$ is $\so$-connected, the converse is also true: $A\subset\ker d_L\alpha$ implies $R_h^*\tilde\alpha=\tilde\alpha$ for all bisections $h$ of $\Gamma_A$, since the group of bisections is generated by the exponentials \cite[Proposition 1.5.8]{mackenzie_general_2005}.
  This proves the equivalence of \ref{lem:forms_on_L-forms_on_GammaLoverGammaA-1} and \ref{lem:forms_on_L-forms_on_GammaLoverGammaA-2} with \ref{lem:forms_on_L-forms_on_GammaLoverGammaA-1prime} in the $\so$-connected case.

  To prove the last assertion of the Lemma, assume that conditions \ref{lem:forms_on_L-forms_on_GammaLoverGammaA-1} and \ref{lem:forms_on_L-forms_on_GammaLoverGammaA-2} are satisfied.
  From \cite[pp.\ 166--167]{weinstein_extensions_1991}, $\alpha$ is $d_L$-closed if and only if $\tilde\alpha$ is closed.
  Since $\pi$ is a surjective submersion and $\tilde\alpha=\pi^*\ol\alpha$, we have that $d\tilde\alpha=0$ if and only if $d\ol\alpha=0$.
\end{proof}

\begin{proof}[Proof of Theorem~\ref{thm:bijection-sympl-str-LA-and-fibrewise-sympl-str-GLGA}]
  Let $\omega$ be a symplectic structure on $(L,A)$. 
  Then $d_L(p^*\omega)=0$ and therefore $A\subset \ker(p^*\omega)\cap\ker d_L(p^*\omega)$, 
  and by Lemma~\ref{lem:forms_on_L-forms_on_GammaLoverGammaA} the left-invariant 2-form $\widetilde{p^*\omega}$ descends to a closed $\Gamma_L$-invariant 2-form $\alpha=\ol{p^*\omega}$ on $\Gamma_L/\Gamma_A$.
  This form is everywhere non-degenerate since it is so on the unit space and it is $\Gamma_L$-invariant.

  Conversely, let $\alpha$ be a $\Gamma_L$-invariant closed non-degenerate fibrewise 2-form on $\Gamma_L/\Gamma_A$.
  Consider the left-invariant 2-form $\pi^*\alpha$ on $\Gamma_L$ and denote by $\Omega$ the corresponding 2-form on $L$.
  Then by Lemma~\ref{lem:forms_on_L-forms_on_GammaLoverGammaA}, $A\subset \ker\Omega$ and $d_L(\Omega)=0$, and by non-degeneracy $A=\ker\Omega$.
  Hence there exists a unique non-degenerate 2-form $\omega$ on $L/A$ with $p^*\omega=\Omega$ and thus $d_L(p^*\omega)=0$.

  The two constructions are clearly inverses of each other.
\end{proof}


\section{Atiyah classes}
\label{sec:atiyah-classes}

In the next two subsections, we adapt to Lie pairs some well-known notions and results about connections and symplectic connections on manifolds.
In the third subsection, we use these to show some properties of the Atiyah class of a symplectic Lie pair, analogous to those found by Kapranov \cite{kapranov_rozanskywitten_1999} for holomorphic symplectic manifolds.


\subsection{Connections}
\label{ssec:connections}

Let $(L,A)$ be a Lie pair.

\begin{defn}
\label{def:torsion-curvature-compatible}
  Let $\nabla$ be an $L$-connection on $L/A$.
  \begin{enumerate}
    \item Its \emph{torsion} and \emph{curvature} tensors $T:L\otimes L\to L/A$ and $R:L\otimes L\to\End(L/A)$ are defined on sections by
      \begin{align*}
        T({l_1},{l_2}) &= \nabla_{l_1}\ol{l_2} - \nabla_{l_2}\ol{l_1} - \ol{[l_1,l_2]} ,
        \\
        R({l_1},{l_2}) &= \nabla_{l_1}\nabla_{l_2} - \nabla_{l_2}\nabla_{l_1} - \nabla_{[l_1,l_2]} ,
      \end{align*}
      for all $l_1,l_2\in\Gamma(L)$.

    \item The connection is said 
      \begin{enumerate}
        \item to \emph{extend the $A$-action} if $T$ descends to a map 
          \[
            \ol{T}:L/A\otimes L/A\to L/A,
          \]
        \item and to be \emph{$A$-compatible} if in addition $R$ descends to a map 
          \[
            \ol{R}:L/A\otimes L/A\to\End(L/A).
          \]
      \end{enumerate}
  \end{enumerate}
\end{defn}

In other words, a connection $\nabla$ is $A$-compatible if and only if
\begin{equation}
\label{eq:A-compatible-CSX}
  \begin{aligned}
    T(a,l)&=\nabla_a \ol{l} - \ol{[a,l]} = 0, \\
    R(a,l)&=\nabla_a \nabla_l - \nabla_l \nabla_a - \nabla_{[a,l]} = 0 ,
  \end{aligned}
\end{equation}
for all $a\in\Gamma(A)$ and $l\in \Gamma(L)$.

Lie pairs and $A$-compatible connections were introduced in \cite{chen_atiyah_2012} in terms of the two conditions in \eqref{eq:A-compatible-CSX}. 
If there existed a quotient $\Gamma_L/\Gamma_A$ integrating $L/A$, and if $\nabla$ induced a $\Gamma_L$-invariant fibrewise affine connection $\ol \nabla$ on $\Gamma_L/\Gamma_A$, then the restriction of its torsion and curvature tensors to the unit manifold $M\subset \Gamma_L/\Gamma_A$ would be the tensors $\ol T$ and $\ol R$, respectively.
The definition of $A$-compatibility in Definition~\ref{def:torsion-curvature-compatible} results from the observation that the two conditions in \eqref{eq:A-compatible-CSX} may be seen as the obstructions to the existence of those tensorial objects.

It follows from \cite{laurent-gengoux_invariant_2014} that the existence of these objects is a sufficient condition for the existence of $\ol \nabla$.
However, in what follows, we will be interested in the case where $A$-compatible connections do not exist, as only this case may lead to non-trivial weight systems.

Connections extending the $A$-action always exist \cite[Lemma 12]{chen_atiyah_2012}, and so do torsion-free connections:

\begin{prop}
  Let $\nabla$ be an $L$-connection on $L/A$ extending the $A$-action.
  Then, there exists a unique $L$-connection $\nabla'$ on $L/A$ extending the $A$-action which is torsion-free and such that $\phi$ is antisymmetric.
  Here, $\phi:L\otimes L\to L/A$ is the bundle map defined on sections by $\phi(l_1,l_2)=\nabla_{l_1}\ol{l_2}-\nabla'_{l_1}\ol{l_2}$.
\end{prop}
\begin{proof}
  Let $\nabla'$ be any other $L$-connection on $L/A$ extending the $A$-action
  and let $T$, $T'$ be the respective torsions of $\nabla$, $\nabla'$.
  Let $\phi$ be as defined in the statement.
  Then $\phi(l_1,l_2)-\phi(l_2,l_1)=T(l_1,l_2)-T'(l_1,l_2)$.
  Hence we have
  \begin{align*}
    \begin{cases}
      T'=0 \\
      \phi\text{ is antisymmetric}
    \end{cases}
    \qquad\Leftrightarrow\qquad
    \phi=\frac{1}{2}T
    \qquad\Leftrightarrow\qquad
    \nabla'_{l_1}\ol{l_2}=\nabla_{l_1}\ol{l_2}-\frac{1}{2}T(l_1,l_2) . 
  \end{align*}
\end{proof}


\subsection{Symplectic connections}
\label{ssec:symplectic-connections}

Let $(L,A)$ be a Lie pair and $\omega\in\Gamma(\Lambda^2(L/A)^*)$ be a non-degenerate two-form.

\begin{defn}
  An $L$-connection on $L/A$ extending the $A$-action is said to be \emph{symplectic} with respect to $\omega$ if it is torsion-free and $\omega$ is parallel, i.e., if $T=0$ and $\nabla\omega=0$.
  Here, 
  \[
    (\nabla_{l_1}\omega)(\ol{l_2},\ol{l_3})
    =
    \rho(l_1)\left(\omega(\ol{l_2},\ol{l_3})\right)
    -\omega(\nabla_{l_1}\ol{l_2},\ol{l_3})
    -\omega(\ol{l_2},\nabla_{l_1}\ol{l_3})
  \]
  for all $l_1,l_2,l_3\in\Gamma(L)$.
\end{defn}

Notice that since symplectic connections $\nabla$ extend the $A$-action, we have 
\begin{align}
\label{eq:relation-closedness-A-invariance-1}
  \nabla_a\omega &= (\partial^A\omega)(a) \\
  \label{eq:relation-closedness-A-invariance-2}
  p^*(\nabla_a\omega) &= \iota_a d_L(p^*\omega)  
\end{align}
for all $a\in\Gamma(A)$.

\begin{prop}
\label{prop:existence-symplectic-connections}
  \begin{enumerate}
    \item \label{prop:existence-symplectic-connections-1}
      There exists a symplectic $L$-connection on $L/A$ if and only if $d_L(p^*\omega)=0$.
    \item \label{prop:existence-symplectic-connections-2}
      When $d_L(p^*\omega)=0$, the set of symplectic connections on the symplectic Lie pair $(L,A,\omega)$ is an affine space over $\Gamma(S^3((L/A)^*))$.
  \end{enumerate}
\end{prop}
\begin{proof}
  A direct computation shows that, 
  for any torsion-free $L$-connection $\nabla$ on $L/A$ and for all $l_1,l_2,l_3\in \Gamma(L)$,
  \[
    d_L(p^*\omega)(l_1,l_2,l_3)
    =
    (\nabla_{l_1}\omega)(\ol{l_2},\ol{l_3})
    +
    (\nabla_{l_2}\omega)(\ol{l_3},\ol{l_1})
    +
    (\nabla_{l_3}\omega)(\ol{l_1},\ol{l_2}) .
  \]
  Hence, if there exists a symplectic connection, then $d_L(p^*\omega)=0$.
  Conversely, assume $d_L(p^*\omega)=0$. Let $\nabla$ be a torsion-free connection on $L/A$ and set
  \[
    \nabla'_{l_1}\ol{l_2} = \nabla_{l_1}\ol{l_2} + S(\ol{l_1},\ol{l_2})
  \]
  with $S:L/A\otimes L/A\to L/A$ defined by 
  \[
    S(\ol{l_1},\ol{l_2}) = \frac{1}{3} (\omega^\flat)^{-1} {\left( \iota_{\ol{l_2}}\nabla_{l_1}\omega + \iota_{\ol{l_1}}\nabla_{l_2}\omega \right)} .
  \]
  Here, $\omega^\flat:L/A\to(L/A)^*$ is the bundle map associated to $\omega$.
  From Eq.~\eqref{eq:relation-closedness-A-invariance-2}, it follows that $\nabla_a\omega=0$ for all $a\in\Gamma(A)$.
  Hence, $S$ is well-defined.
  It is then easy to check that $\nabla'$ is still torsion-free and that $\nabla'\omega=\frac{1}{3}d_L(p^*\omega)=0$. Hence, $\nabla'$ is a symplectic connection.
  This proves point \ref{prop:existence-symplectic-connections-1}.

  Let now $\nabla$ be a symplectic $L$-connection on $L/A$ and $\nabla'$ an $L$-connection on $L/A$ extending the $A$-action. 
  Since they both extend the $A$-action, their difference $\nabla-\nabla'$ defines a tensor $\phi\in \Gamma((L/A)^* \otimes (L/A)^* \otimes L/A)$.
  Then $\nabla'$ is torsion-free if and only if $\phi\in \Gamma(S^2((L/A)^*)\otimes L/A)$.
  And since $\nabla\omega=0$, the condition $\nabla'\omega=0$ is equivalent to
  \begin{dmath}
  \label{eq:difference-of-nablaomega}
    0 
    = (\nabla_{l_1}\omega)(\ol{l_2}, \ol{l_3}) - (\nabla'_{l_1}\omega)(\ol{l_2}, \ol{l_3})
    = -\omega(\phi(\ol{l_1},\ol{l_2}),\ol{l_3}) + \omega(\phi(\ol{l_1},\ol{l_3}),\ol{l_2}) 
    = -({\id}\otimes{\id}\otimes\omega^\flat)(\phi)(\ol{l_1},\ol{l_2},\ol{l_3})
       +({\id}\otimes{\id}\otimes\omega^\flat)(\phi)(\ol{l_1},\ol{l_3},\ol{l_2}) ,
  \end{dmath}
  i.e., to $\psi=({\id}\otimes{\id}\otimes \omega^\flat)(\phi)$ being in $\Gamma((L/A)^*\otimes S^2((L/A)^*))$.
  Hence, $\nabla'$ is symplectic if and only if $\psi \in \Gamma(S^3((L/A)^*))$.
  This proves point \ref{prop:existence-symplectic-connections-2}.
\end{proof}


\subsection{Symmetries of the Atiyah cocycle}
\label{ssec:symmetries-of-the-atiyah-cocycle}

Let $(L,A)$ be a Lie pair.
The curvature
\begin{align}
\label{eq:curvature}
  R(l_1,l_2)\ol{l_3} 
  = 
  \nabla_{l_1} \nabla_{l_2} \ol{l_3} - \nabla_{l_2} \nabla_{l_1} \ol{l_3} - \nabla_{[l_1,l_2]} \ol{l_3}
\end{align}
of an $L$-connection $\nabla$ on $L/A$ \emph{extending the $A$-action} vanishes on $A\wedge A$.
Hence, restricting its first argument to $A$ defines a 1-form
\[
  R^\nabla \in \Omega^1(A,(L/A)^* \otimes (L/A)^* \otimes L/A)
\]
by $R^\nabla(a)(\ol{l_1},\ol{l_2})=R(a,l_1)l_2$ for all $a\in \Gamma(A)$ and $l_1,l_2\in\Gamma(L)$.
This 1-form on $A$ is $\partial^A$-closed \cite[Theorem 16]{chen_atiyah_2012}, and is called the \emph{Atiyah cocycle} associated to $\nabla$.
The induced 1-cohomology class $\alpha_{L/A}=[R^\nabla]\in H^1(A,(L/A)^*\otimes (L/A)^* \otimes L/A)$ is, moreover, independent of the connection and is called the \emph{Atiyah class of the Lie pair} $(L,A)$.

Let $\omega\in\Gamma(\Lambda^2(L/A)^*)$ be a non-degenerate two-form.
Using the isomorphism 
$\omega^\flat:L/A\to (L/A)^*$ induced by $\omega$, $R^\nabla$ defines another 1-form
\[
  \tilde R^\nabla\in\Omega^1(A,(L/A)^* \otimes (L/A)^*\otimes (L/A)^*)
\]
by
$\tilde R^\nabla = \left(\id \otimes \id \operatorname{\otimes} \omega^\flat \right)\circ R^\nabla$.

It was proved in \cite[Proposition 55]{chen_atiyah_2012} that in fact $\alpha_{L/A}\in H^1(A,S^2((L/A)^*) \otimes L/A)$.
Below, we give another proof of that result and show that, as in the holomorphic symplectic case \cite{rozansky_hyper-kahler_1997,kapranov_rozanskywitten_1999}, the image of the 1-cocycle $\tilde R^\nabla$ (not only the cohomology class) is completely symmetric if $\nabla$ is symplectic.
This stems from the fact that the curvature of a symplectic connection acts by symplectic endomorphisms.

Let $\tau_{(12)},\tau_{(23)}:((L/A)^*)^{\otimes 3}\to ((L/A)^*)^{\otimes 3}$ be, respectively, the swapping of the first and last two components (see \eqref{eq:definition-tau} in Section~\ref{sec:preliminaries}).
\begin{lem}
\label{lem:symmetries-atiyah-cocycle}
  Let $\nabla$ be an $L$-connection on $L/A$ extending the $A$-action. Then,
  \begin{enumerate}[label=(\arabic*)]
    \item \label{lem:symmetries-atiyah-cocycle-1}
      $\left[\tilde R^\nabla-\tau_{(23)}\circ \tilde R^\nabla\right]$ is independent of $\nabla$ if $\partial^A\omega=0$;
    \item \label{lem:symmetries-atiyah-cocycle-2}
      $\tilde R^\nabla-\tau_{(23)}\circ \tilde R^\nabla=\partial^A (-\nabla\omega)$
      if $d_L(p^*\omega)=0$;
    \item \label{lem:symmetries-atiyah-cocycle-3}
      $R^\nabla - \tau_{(12)}\circ R^\nabla = \partial^A (T)$.
  \end{enumerate}
\end{lem}

\begin{proof}
  First notice that if $E$ and $F$ are two $A$-modules, and $f:E\to F$ is an $A$-module map, then 
  $f^*:\Omega^\bullet(A,E)\to\Omega^\bullet(A,F)$ is a cochain map.
  Now the map $\omega^\flat:L/A\to (L/A)^*$ is an $A$-module map since $\partial^A\omega=0$.
  Indeed, we have $\nabla^A_a\circ \omega^\flat - \omega^\flat \circ \nabla^A_a = \partial^A_a\omega$ for all $a\in \Gamma(A)$. 
  Here, $\nabla^A$ is extended as in Eq.~\eqref{eq:extension-of-connection-to-tensor-powers}.

  Moreover, if $\nabla'$ is another $L$-connection on $L/A$ extending the $A$-action then, writing $\phi=\nabla-\nabla'$, we get $R^\nabla-R^{\nabla'}=\partial^A\phi$ \cite[Theorem 16(b)]{chen_atiyah_2012}.

  It follows from the two foregoing paragraphs that, if $\partial^A\omega=0$, the cohomology class of $\tilde R^\nabla=({\id}\otimes{\id}\otimes{\omega^\flat})\circ R^\nabla$ is independent of $\nabla$.
  Since $\tau_{(23)}$ is also an $A$-module map, the same holds for $\tilde R^\nabla-\tau_{(23)}\circ \tilde R^\nabla$, proving the first item.

  Actually, we have
  \begin{dmath}
    (\tilde R^\nabla-\tau_{(23)}\circ \tilde R^\nabla)-(\tilde R^{\nabla'}-\tau_{(23)}\circ \tilde R^{\nabla'})
    = (\tilde R^\nabla - \tilde R^{\nabla'}) - \tau_{(23)} \circ (\tilde R^\nabla - \tilde R^{\nabla'})
    = ({\id} - \tau_{(23)}) ({\id}\otimes{\id}\otimes\omega^\flat)(\partial^A\phi)
    = \partial^A(({\id} - \tau_{(23)}) ({\id}\otimes{\id}\otimes\omega^\flat)(\phi))
    \label{eq:sym23-of-tildeRnabla} = \partial^A(-\nabla\omega+\nabla'\omega) ,
  \end{dmath}
  where the last step follows from \eqref{eq:difference-of-nablaomega}.

  Now, if $\omega$ is closed, there exists a symplectic connection $\nabla'$, for which the curvature $R'$ (see Eq.~\eqref{eq:curvature}) is thus a symplectic endomorphism:
  \begin{align}
  \label{eq:curvature-is-symplectic-endomorphism}
    \omega(R'(l_1,l_2)\ol{l_3},\ol{l_4}) + \omega(\ol{l_3},R'(l_1,l_2)\ol{l_4}) = 0,
  \end{align}
  for all $l_1,l_2,l_3,l_4\in\Gamma(L)$. 
  Taking $l_1\in\Gamma(A)$, Eq.~\eqref{eq:curvature-is-symplectic-endomorphism} yields $\tilde R^{\nabla'}-\tau_{(23)}\circ \tilde R^{\nabla'}=0$. Inserting this last equation and $\nabla'\omega=0$ in 
  Eq.~\eqref{eq:sym23-of-tildeRnabla},
  we get $\tilde R^\nabla-\tau_{(23)}\circ \tilde R^\nabla=\partial^A(-\nabla\omega)$ for any $\nabla$ extending the $A$-action.
  This proves the second item.

  Let us now prove the last item.
  The curvature of $\nabla$ satisfies the algebraic Bianchi identity
  \[
    R(l_1,l_2)\ol{l_3}
    - T(l_1,T(l_2,l_3)) - (\nabla_{l_1}T)(l_2,l_3)
    +\text{cycl.\ perm.}
    = 0
    .
  \]
  Taking $l_3=a\in \Gamma(A)$ yields $R(a,l_1)\ol{l_2} - R(a,l_2)\ol{l_1} = (\nabla_aT)(l_1,l_2)$, i.e., $R^\nabla - \tau_{(12)}\circ R^\nabla = \partial^A (T)$.
\end{proof}

The following Theorem directly follows from items \ref{lem:symmetries-atiyah-cocycle-2} and \ref{lem:symmetries-atiyah-cocycle-3} of the previous Lemma.

\begin{thm}
  Let $(L,A,\omega)$ be a symplectic Lie pair, and $\nabla$ an $L$-connection on $L/A$ extending the $A$-action. Then, 
  \begin{enumerate}[label=(\arabic*)]
    \item the Atiyah class%
      \footnote{We also call \emph{Atiyah class} the composition $\tilde\alpha_{L/A}=({\id}\otimes{\id}\otimes\omega^\flat)\circ\alpha_{L/A}$ of $\alpha_{L/A}$ with the $A$-module morphism ${\id}\otimes{\id}\otimes\omega^\flat$.}
      $\tilde\alpha_{L/A}=[\tilde R^\nabla]$ of $(L,A,\omega)$ lies in $H^1(A,S^3((L/A)^*))$, and
    \item when $\nabla$ is symplectic, the 1-cocycle $\tilde R^\nabla$ itself has totally symmetric image.
  \end{enumerate}
\end{thm}

To conclude this section, we quote a result from \cite{chen_atiyah_2012} which we will need in the next section.
The result is the cocycle version of the fact that the Atiyah class $\alpha_{L/A}$ is a kind of Lie bracket which turns $L/A[-1]$ into a Lie algebra object in the bounded derived category of $U(A)$-modules (see \cite{roberts_rozanskywitten_2010,laurent-gengoux_kapranov_2014} for more information about that interpretation).
Denote by $R_2:L/A\otimes L/A \to A^* \otimes L/A$ the bundle map defined by $R_2(b_1,b_2)(a)=R^\nabla(a)(b_1,b_2)$ for all $a\in\Gamma(A)$ and $b_1,b_2\in\Gamma(L/A)$.

\begin{thm}[{\cite[Corollary 71]{chen_atiyah_2012}}]
\label{thm:Corollary71-from-CSX}
  Define a 2-form $\psi\in\Omega^2(A,((L/A)^*)^{\otimes 3} \otimes L/A)$ by setting, for all $b_1,b_2,b_3\in\Gamma(L/A)$,
  \[
    \psi(b_1,b_2,b_3)
    =
    \left\lceil R_2(b_1,R_2(b_2,b_3)) \right\rfloor 
    + \left\lceil R_2(R_2(b_1,b_2),b_3) \right\rfloor 
    + \left\lceil R_2(b_2,R_2(b_1,b_3)) \right\rfloor
    .
  \]
  Here, the notation $\left\lceil R_2(b_1,R_2(b_2,b_3)) \right\rfloor \in \Omega^2(A,L/A)$ means taking the wedge product on $A$-forms and the composition ``as written'' on the $(L/A)^*$-parts, i.e., 
  \[
    \lceil R_2(b_1,R_2(b_2,b_3)) \rfloor (a_1,a_2) 
    = 
    R^\nabla(a_1)(b_1,R^\nabla(a_2)(b_2,b_3)) - R^\nabla(a_2)(b_1,R^\nabla(a_1)(b_2,b_3)) 
  \]
  for all $a_1,a_2\in\Gamma(A)$ and $b_1,b_2,b_3\in\Gamma(L/A)$.

  Then $\psi$ is a coboundary, i.e., $\psi=\partial^A\phi$ for some $\phi\in\Omega^1(A,((L/A)^*)^{\otimes 3} \otimes L/A)$.
\end{thm}


\section{Weight systems}
\label{sec:weight-systems}

In this section, we follow Rozansky and Witten \cite{rozansky_hyper-kahler_1997}, Kapranov \cite{kapranov_rozanskywitten_1999}, and Sawon \cite{sawon_rozanskywitten_1999} to build, from any symplectic Lie pair $(L,A,\omega)$, a weight system on trivalent diagrams with values in the Lie algebroid cohomology of $A$.


\subsection{Trivalent diagrams}
This subsection contains a recollection of notions about trivalent diagrams.
We refer to Bar-Natan \cite{bar-natan_vassiliev_1995} for more information and references about the subject.
We briefly introduce trivalent graphs and two notions of orientation on them. 
We then follow Kapranov \cite{kapranov_rozanskywitten_1999} to compare the two notions of orientation.

\begin{defn}
\label{def:trivalent-graph}
  A \emph{trivalent graph} is a (possibly disconnected) finite graph having only trivalent vertices.
\end{defn}

We will need two different, but equivalent, notions of orientation of a trivalent graph, one called linear orientation and the other cyclic orientation.
Both are defined by equivalence relations, which we describe now.

\begin{defn}
\label{def:linear-orientation-trivalent-graph}
  \begin{enumerate}
    \item A \emph{representative of a linear orientation} of a trivalent graph is a pair composed of an orientation of the edges and a linear ordering of the vertices (Fig.~\ref{fig:examples-orientation}).

    \item Two linear orientation representatives differing on $n$ edges and by a permutation $\sigma$ of the ordering of the vertices are considered \emph{equivalent} if $(-1)^n=\sign(\sigma)$.

    \item A \emph{linear orientation} of a trivalent graph is an equivalence class of linear orientation representatives.
  \end{enumerate}
\end{defn}

\begin{defn}
\label{def:cyclic-orientation-trivalent-graph}
  \begin{enumerate}
    \item A \emph{representative of a cyclic orientation} of a trivalent graph is the data, at each vertex, of a cyclic ordering (i.e., in a planar representation of the graph around that vertex, a choice of clockwise or anticlockwise ordering) of the three flags at that vertex (Fig.~\ref{fig:examples-orientation}).

    \item Two cyclic orientation representatives are considered \emph{equivalent} if they differ at an even number of vertices.

    \item A \emph{cyclic orientation} of a trivalent graph is an equivalence class of cyclic orientation representatives.
  \end{enumerate}
\end{defn}

\begin{figure}[h]
\label{fig:examples-orientation}
  \begin{center}
    ~
    \hfill
    \includegraphics{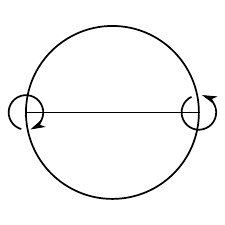}
    \hfill
    \includegraphics{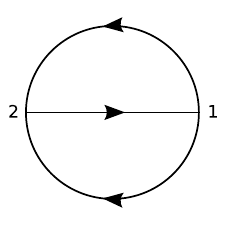}
    \hfill
    ~
  \end{center}
  \caption{Representatives of a cyclic (on the left) and a linear (on the right) orientations of the $\Theta$ graph.}
\end{figure}

\begin{rem}
  When representing a trivalent diagram in the plane, if no orientation is specified, 
  by default it is assumed to have the cyclic orientation defined by the anticlockwise ordering at each vertex.
\end{rem}

It is readily seen that a trivalent graph always satisfies $3\,\#\{\text{vertices}\}=2\,\#\{\text{edges}\}$.
Since its number of vertices is thus even, the \emph{order} of a graph may be defined as half the number of vertices.
A \emph{flag} is a pair comprising a vertex and one of the three edges attached to that vertex.

Let us borrow some more notation from \cite[Sec. 5.2]{kapranov_rozanskywitten_1999} (see also the references therein).
\begin{enumerate}
  \item For a finite set $S$, denote by $\fR^S$ the real vector space generated by the elements of $S$, and by $\det(S)$ the top exterior power of $\fR^S$. 
  \item For a trivalent graph $\gamma$, denote by $V(\gamma)$, $E(\gamma)$, $F(\gamma)$ the sets of vertices, edges and flags of $\gamma$.
    Let $F_v(\gamma)$ be the set of the three flags with vertex $v$, and $F^e(\gamma)$ be the set of the two flags with edge $e$. 
\end{enumerate}
With this notation, the choice of a direction (i.e., of a non-zero element) in the 1-dimensional space $\det(F_v(\gamma))$ is a choice of cyclic ordering at the vertex $v$, and similarly  a direction in $\det(F^e(\gamma))$ is an orientation of the edge $e$.

Since a cyclic orientation is the choice of a cyclic ordering at each vertex (i.e., a direction in $\det(F_v(\gamma))$ for every vertex $v$), two sets of choices being equivalent if the directions differ at an even number of vertices, we obtain the following Lemma.
\begin{lem}
\label{lem:cyclic-orientation}
  A \emph{cyclic orientation} on a graph $\gamma$ is a choice of direction in the 1-dimensional space 
  \[
    \bigotimes_{v\in V(\gamma)} \det(F_v(\gamma)) . 
  \]
\end{lem}

Similarly, a linear orientation is the choice of a linear ordering on the set $V(\gamma)$ of vertices and an orientation of each edge (that is, a direction in $\det(F^e(\gamma))$ for every edge $e$), up to the equivalence relation of Definition~\ref{def:linear-orientation-trivalent-graph}.
Hence,
\begin{lem}
\label{lem:linear-orientation}
  A \emph{linear orientation} is a choice of direction in the 1-dimensional space 
  \[
  \det(V(\gamma))\otimes \bigotimes_{e\in E(\gamma)} \det(F^e(\gamma)) . 
\]
\end{lem}

The following result is proved in \cite[Lem.\ 5.3.2]{kapranov_rozanskywitten_1999} (see also \cite[p.\ 20]{sawon_rozanskywitten_1999}).
\begin{prop}
\label{prop:kapranov-orientations}
  The linear and the cyclic orientations coincide.
  More precisely, for every trivalent graph $\gamma$, there is 
  a canonical identification of 1-dimensional vector spaces
  \[
    \det(V(\gamma))\otimes \bigotimes_{e\in E(\gamma)} \det(F^e(\gamma)) 
    \cong \bigotimes_{v\in V(\gamma)}\det(F_v(\gamma)) , 
  \]
  i.e., 
  a canonical correspondence between the linear and cyclic orientations.
\end{prop}
Following the proof in the two foregoing references, let us briefly show how to compute one orientation from the other.
Consider a trivalent graph $\gamma$ of order $k$.

We will show that both sides of the isomorphism of Proposition~\ref{prop:kapranov-orientations} are canonically isomorphic to a third 1-dimensional space, namely $\det(V(\gamma))\otimes\det(F(\gamma))$, and will exhibit explicit isomorphisms. 

The first of these isomorphisms is the linear map 
\[
  \alpha:\bigotimes_{v\in V(\gamma)} \det(F_v(\gamma)) \to \det(V(\gamma))\otimes\det(F(\gamma))
\]
defined as follows. Choose a linear ordering on the set of vertices, so that they are numbered $v_1,\dots,v_{2k}$. Then set
\[
  \alpha\left(\bigotimes_{v\in V(\gamma)} d_v\right) 
  = (v_1\wedge\dots\wedge v_{2k}) \otimes (d_{v_1} \wedge \dots \wedge d_{v_{2k}})
\]
for all $d_v\in\det(F_v(\gamma))$, $v\in V(\gamma)$.
The map $\alpha$ is an isomorphism since the subspaces $\fR^{F_v(\gamma)}$ are in direct sum in $\fR^{F(\gamma)}$.
Since $d_{v_i}\wedge d_{v_j} = - d_{v_j}\wedge d_{v_i}$ and $v_i \wedge v_j = - v_j\wedge v_i$ for all $i,j=1,\dots,2k$, it is clear that $\alpha$ is independent of the chosen ordering of the vertices.

To construct the second isomorphism, define 
\[
  \beta:\bigotimes_{e\in E(\gamma)} \det(F^e(\gamma)) \to \det(F(\gamma))
\]
as follows.
Choose a linear ordering on the set of edges, so that they are numbered $e_1,\dots,e_{3k}$. Then set
\[
  \beta\left(\bigotimes_{e\in E(\gamma)}d^e\right) = d^{e_1} \wedge \dots \wedge d^{e_{3k}}
\]
for all $d^e\in \det(F^e(\gamma))$, $e\in E(\gamma)$.
Again, the map $\beta$ is an isomorphism since the subspaces $\fR^{F^e(\gamma)}$ are in direct sum in $\fR^{F(\gamma)}$.
And since $d^{e_i}\wedge d^{e_j} = d^{e_j}\wedge d^{e_i}$ for all $i,j=1,\dots,3k$, the map $\beta$ is independent of the chosen ordering on $E(\gamma)$.

Now the map
\[
  (\id\otimes \beta^{-1}) \circ \alpha :
  \bigotimes_{v\in V(\gamma)}\det(F_v(\gamma))
  \to
  \det(V(\gamma))\otimes \bigotimes_{e\in E(\gamma)} \det(F^e(\gamma))
\]
is the isomorphism in Proposition~\ref{prop:kapranov-orientations}.

\begin{ex}
  Concretely, consider the following diagram $\gamma$ with the cyclic orientation given by the default anticlockwise ordering at each vertex,
  \[
    \vcenter{\hbox{\includegraphics{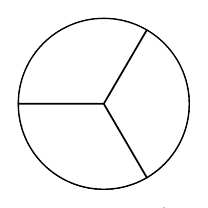}}} .
  \]
  We label the flags $f_1,\dots,f_{12}$ and the vertices $v_1,\dots,v_4$ as follows
  \[
    \vcenter{\hbox{\includegraphics[scale=1.2]{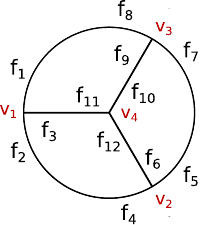}}} .
  \]
  The cyclic orientation on $\gamma$ thus corresponds to the direction given by the following element
  \begin{dmath}
  \label{eq:example-cyclic-orientation}
    (f_1\wedge f_2\wedge f_3)
    \otimes (f_4\wedge f_5\wedge f_6) 
    \otimes (f_7\wedge f_8\wedge f_9)
    \otimes (f_{10}\wedge f_{11}\wedge f_{12})
  \end{dmath}
  in $\bigotimes_{v\in V(\gamma)}\det(F_v(\gamma))$.

  Applying $\alpha$ to this element, we get 
  \begin{dmath*}
    (v_1\wedge v_2\wedge v_3\wedge v_4) 
    \otimes (f_1\wedge f_2\wedge f_3\wedge f_4\wedge f_5\wedge f_6\wedge f_7\wedge f_8\wedge f_9\wedge f_{10}\wedge f_{11}\wedge f_{12}) .
  \end{dmath*}
  Applying $\id\otimes\beta^{-1}$, we arrive at\footnote{Here, in a similar way to what we did in Eq.~\ref{eq:example-cyclic-orientation}, for the ease of presentation we represented an element of the unordered tensor product $\bigotimes_{e\in E(\gamma)} \det(F^e(\gamma))$ as $d_1\otimes d_2\otimes \cdots \otimes d_{3k}$ for some $d_i\in\det(F^{e_i})$, by choosing an arbitrary linear ordering $e_1, e_2, \dots, e_{3k}$ on the set of edges.}
  \begin{dmath*}
    (v_1\wedge v_2\wedge v_3\wedge v_4) 
    \otimes \Big( (f_1\wedge f_8)\otimes (f_4\wedge f_2)\otimes (f_{11}\wedge f_3)\otimes (f_7\wedge f_5)\otimes (f_{6}\wedge f_{12}) \otimes (f_9\wedge f_{10}) \Big) .
  \end{dmath*}

  Hence, in this case, the correspondence between the cyclic and linear orientations is
  \[
    \vcenter{\hbox{\includegraphics{Trivalent-example-compiled.pdf}}} \quad \leftrightsquigarrow  \quad
    \vcenter{\hbox{\includegraphics{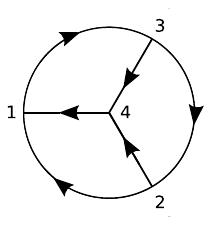}}} ,
  \]
  where on the left we have the implicit anticlockwise ordering at each vertex, and on the right a corresponding representative of the linear orientation.
\exend
\end{ex}

These two equivalent notions of orientation will both be needed, as the linear orientation is more adapted to the construction of weight systems, while the cyclic orientation is better suited for the definition of the IHX and AS relations.


\subsection{Weight systems}
\label{ssec:weight-systems}

We first recall the definition of a weight system on trivalent diagrams \cite{bar-natan_vassiliev_1995,rozansky_hyper-kahler_1997,garoufalidis_finite_1998}.
Denote by $\mathcal{T}$ (respectively, $\mathcal{T}_n$) the linear span (over $\fK$) of all trivalent diagrams (respectively, all trivalent diagrams of order $n$).

\begin{defn}
\label{def:T-weight-system}
  Let $V$ be a vector space over $\fK$. A $V$-valued \emph{$\mathcal{T}$-weight system} of order $n$ is a linear map $w_n:\mathcal T_n\to V$ satisfying
  \begin{enumerate}[label=(\arabic*)]
    \item the AS relation:
    $w_n \left( \vcenter{\hbox{\includegraphics{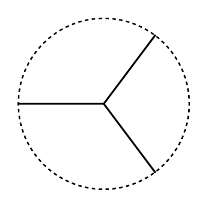}}} - \vcenter{\hbox{\includegraphics{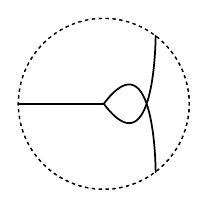}}} \right) = 0$,
    \item the IHX relation:
    $w_n \left( \vcenter{\hbox{\includegraphics{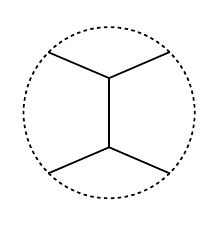}}} - \vcenter{\hbox{\includegraphics{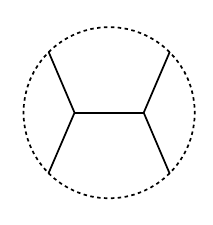}}} + \vcenter{\hbox{\includegraphics{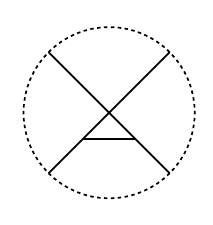}}} \right) = 0$,
  \end{enumerate}
  where, by $\vcenter{\hbox{\includegraphics[scale=.25]{AS1-compiled.pdf}}}$ and $\vcenter{\hbox{\includegraphics[scale=.25]{AS2-compiled.pdf}}}$ in the AS relation, we mean two diagrams which are identical outside a disk and as in the picture inside the disk, and similarly for the IHX relation.

  The diagrams above are taken with their default cyclic orientation given by the anticlockwise ordering at each vertex.
  Moreover, only trivalent vertices appear. Any apparent tetravalent vertex is not a vertex.
\end{defn}

Our aim is, from $(L,A,\omega)$ and for each positive integer $k$ such that $2k<\rk(A)$, to construct an $H^{2k}(A)$-valued $\mathcal{T}$-weight system of order $k$.

The idea of the construction is very simple and closely follows \cite{rozansky_hyper-kahler_1997,kapranov_rozanskywitten_1999} in the holomorphic symplectic case: to each vertex is associated a copy of the fully symmetric Atiyah class $\tilde\alpha_{L/A}$.
The edges are then used to contract all the $(L/A)^*$ parts of these copies with the help of the inverse $\omega^{-1}$ of the symplectic form. After complete antisymmetrization, the result is a scalar-valued $A$-cohomology class.
To describe precisely the procedure and the way the orientation fits in, let us introduce the notion of \emph{$\gamma$-admissible transformation} of $((L/A)^*)^{\otimes 6k}$, for a trivalent diagram $\gamma$.

Let $\gamma$ be a trivalent diagram of order $k$.
A set of ``admissible'' permutations of the $6k$ components of the tensor product $((L/A)^*)^{\otimes 6k}$ is associated to $\gamma$ in the following way.
Each such permutation will be defined by the following choices:
\begin{enumerate}[label=(C\arabic*)]
  \item \label{item:admissible-perm-or-representative} a linear orientation representative of the orientation of $\gamma$, i.e.,
  \begin{enumerate}
    \item \label{item:admissible-perm-lin-ord-vert} a linear ordering $v_1,\dots,v_{2k}$ of the vertices,
    \item \label{item:admissible-perm-or-edges} an orientation of each edge $e$, given by a linear ordering $f^1(e),f^2(e)$ of the flags of that edge;
  \end{enumerate}
  
  \item \label{item:admissible-perm-lin-ord-edges} a linear ordering $e_1,\dots,e_{3k}$ of the edges; 
  
  \item \label{item:admissible-perm-lin-ord-flags-at-vertices} a linear ordering $f_1(v),f_2(v),f_3(v)$ of the flags at each vertex $v$.
\end{enumerate}
From \ref{item:admissible-perm-lin-ord-vert} and \ref{item:admissible-perm-lin-ord-flags-at-vertices}, one gets a linear ordering 
\[
  f_1(v_1),f_2(v_1),f_3(v_1),f_1(v_2),f_2(v_2),f_3(v_2),\dots,f_1(v_{2k}),f_2(v_{2k}),f_3(v_{2k})
\]
of the flags of $\gamma$, and from \ref{item:admissible-perm-or-edges} and \ref{item:admissible-perm-lin-ord-edges}, one gets a second linear ordering 
\[
  f^1(e_1),f^2(e_1),f^1(e_2),f^2(e_2),\dots,f^1(e_{3k}),f^2(e_{3k})
\]
of the flags of $\gamma$.
Together, these orderings define a permutation $\sigma_\gamma$ of $\{1,\dots,6k\}$ sending the first one to the second.
It is given by 
\[
	\sigma_\gamma=\epsilon\circ\nu^{-1},
\] 
where $\nu$ and $\epsilon$ are the bijections defined by
\begin{align*}
  \nu&:F(\gamma)\to\{1,\dots,6k\}:f_i(v_j)\mapsto 3(j-1)+i \\
  \epsilon&:F(\gamma)\to\{1,\dots,6k\}:f^i(e_j)\mapsto 2(j-1)+i.
\end{align*}
Each such permutation defines a linear automorphism $\pi_\gamma$ of $((L/A)^*)^{\otimes 6k}$ by $\pi_\gamma = \tau_{\sigma_\gamma}$ (see \eqref{eq:definition-tau} in Section~\ref{sec:preliminaries}).

\begin{defn}
\label{def:admissible-permutation}
  The transformations $\pi_\gamma$ of $((L/A)^*)^{\otimes 6k}$ obtained from the choices \ref{item:admissible-perm-or-representative}, \ref{item:admissible-perm-lin-ord-edges} and \ref{item:admissible-perm-lin-ord-flags-at-vertices} above are called the \emph{$\gamma$-admissible transformations} of $((L/A)^*)^{\otimes 6k}$.
\end{defn}

In Theorem \ref{thm:T-weight-system}, we will build the weight associated to a symplectic Lie pair $(L,A,\omega)$ and a trivalent diagram $\gamma$ as a composition of three objects canonically associated to $(L,A,\omega)$ and $\gamma$, without any further choice of orientation representatives.
To this end, we make the following definition.

\begin{defn}
\label{def:canonical-permutation}
  The set of all $\gamma$-admissible transformations is denoted by $\Adm(\gamma)$.
  A transformation $\Pi_\gamma$ of $((L/A)^*)^{\otimes 6k}$ is canonically associated to $\gamma$ by
  \[
    \Pi_\gamma = \frac{1}{\lvert\Adm(\gamma)\rvert}\sum_{\pi_\gamma\in\Adm(\gamma)} \pi_\gamma .
  \]
\end{defn}

We can now state and prove the main theorem, generalizing \cite[Section 3.4]{rozansky_hyper-kahler_1997} and \cite[Theorem 5.4, first part]{kapranov_rozanskywitten_1999}.
\begin{thm}
\label{thm:T-weight-system}
  Let $(L,A,\omega)$ be a symplectic Lie pair. 
  For each $k\in\fN$, there is a canonical $\mathcal{T}$-weight system
  \[
	 w_{(L,A,\omega),k} : \mathcal{T}_k \to H^{2k}(A) 
\]
  of order $k$, defined on diagrams $\gamma$ of order $k$ by
  \begin{dmath}
  \label{eq:weight-system-as-composition}
    w_{(L,A,\omega),k}(\gamma) 
    = 
    \left(\omega^{-1}\right)^{\otimes 3k}
    \circ
    \Pi_\gamma
    \circ
    \left(\tilde\alpha_{L/A}\right)^{\abxcup 2k} .
  \end{dmath}
\end{thm}

\begin{proof}
  Let $\gamma$ be a trivalent diagram of order $k$.

  \medskip
  \emph{Step 1.}
  We first define a candidate $w(\gamma)\in\Omega^{2k}(A)$ for a cocycle representing $w_{(L,A,\omega),k}(\gamma)$.

  Let $\nabla$ be a symplectic connection on $(L,A,\omega)$.
  Pick a choice of admissible transformation $\pi_\gamma\in\Adm(\gamma)$, and consider the composition 
  \begin{dmath}
  \label{eq:composition-weight-system-small-pi}
    w(\gamma)=\left(\omega^{-1}\right)^{\otimes 3k}
    \circ
    \pi_\gamma
    \circ
    \left(\tilde R^\nabla\right)^{\wedge 2k} ,
  \end{dmath}
  where $(\tilde R^\nabla)^{\wedge 2k}$ is considered as a map from $\Lambda^{2k}A$ to $((L/A)^{*})^{\otimes 6k}$, $\pi_\gamma$ as a map from $((L/A)^{*})^{\otimes 6k}$ to itself, and $(\omega^{-1})^{\otimes 3k}$ as a map from $((L/A)^{*})^{\otimes 6k}$ to the trivial line bundle $M\times \fK$.

  Since $\tilde R^\nabla$ is a cocycle, and since $\omega^{-1}$ and $\pi_\gamma$ are $A$-module maps, 
  we have $\partial^A(w(\gamma))=0$. 
  Hence, $w(\gamma)$ descends to an element of $H^{2k}(A)$, which we denote by
  \begin{dmath}
  \label{eq:composition-weight-system-small-pi-in-cohomology}
    \ol w(\gamma)
    =
    \left(\omega^{-1}\right)^{\otimes 3k}
    \circ
    \pi_\gamma
    \circ
    \left(\tilde\alpha_{L/A}\right)^{\abxcup 2k} .
  \end{dmath}

  \medskip
  \emph{Step 2.}
  We show that $w(\gamma)$ is independent of the choice of $\gamma$-admissible transformation $\pi_\gamma$.

  Since $\left(\omega^{-1}\right)^{\otimes 3k}$ lies in $\Gamma\left(S^{3k}\left((L/A)^{\otimes 2}\right)\right)$, and since the values of $(\tilde R^\nabla)^{\wedge 2k}$ lie in the subbundle $\left(S^3((L/A)^*)\right)^{\otimes 2k}$, it is clear that \eqref{eq:composition-weight-system-small-pi} is independent of the choices \ref{item:admissible-perm-lin-ord-edges} and \ref{item:admissible-perm-lin-ord-flags-at-vertices}, respectively.

  Similarly, if two $\gamma$-admissible transformations $\pi_\gamma$ and $\pi'_\gamma$ only differ by \ref{item:admissible-perm-lin-ord-vert}, i.e., by a permutation $\sigma$ of the vertices, then $\pi'_\gamma = \pi_\gamma\circ \pi$ with $\pi$ being the transformation permuting by $\sigma$ the $((L/A)^*)^{\otimes 3}$-components of $((L/A)^*)^{\otimes 6k}$.
  But 
  \begin{dmath}
  \label{eq:composition-sign-C1a}
    \pi\circ (\tilde R^\nabla)^{\wedge 2k} 
    = 
    (-1)^{\sign(\sigma)} (\tilde R^\nabla)^{\wedge 2k} .
  \end{dmath}
  On the other hand, if $\pi_\gamma$ and $\pi'_\gamma$ only differ by \ref{item:admissible-perm-or-edges}, then $\pi'_\gamma=\pi\circ\pi_\gamma$ with $\pi$ swapping the two components of, say, a number $n$ of $((L/A)^*)^{\otimes 2}$-components. The antisymmetry of $\omega^{-1}$ then yields 
  \begin{dmath}
  \label{eq:composition-sign-C1b}
    \left(\omega^{-1}\right)^{\otimes 3k} \circ \pi= (-1)^n \left(\omega^{-1}\right)^{\otimes 3k}.
  \end{dmath}
  This proves that \eqref{eq:composition-weight-system-small-pi} is independent of the choice \ref{item:admissible-perm-or-representative} of orientation representative.

  Overall, this shows that \eqref{eq:composition-weight-system-small-pi} and \eqref{eq:composition-weight-system-small-pi-in-cohomology} are independent of the choice of $\pi_\gamma\in\Adm(\gamma)$, so that we may average over all such choices to get $w_{(L,A,\omega),k}(\gamma)$ as a composition \eqref{eq:weight-system-as-composition} of three objects canonically associated to $(L,A,\omega)$ and $\gamma$.

  \medskip
  \emph{Step 3.}
  We prove the AS and IHX relations.

  That the AS relations holds follows from Eq.~\eqref{eq:composition-sign-C1a} and \eqref{eq:composition-sign-C1b}: a change of orientation of $\gamma$ induces a change of sign of $w_{(L,A,\omega),k}(\gamma)$.

  Assume that three trivalent diagrams $\gamma_I$, $\gamma_H$, and $\gamma_X$ are identical except in some disk containing two vertices, where they respectively look like the three diagrams in the definition of the IHX relation (Definition~\ref{def:T-weight-system}).
  The  I, H, and X parts (i.e., the parts inside the disks) are assumed to have the default anticlockwise ordering at their vertices, 
  and each vertex outside the disk is assumed to have the same cyclic ordering in all three diagrams.
  As a result, the three diagrams each come with a fixed cyclic orientation representative.

  To compute the IHX relation, we need to show that
  \[
    w(\gamma_I) - w(\gamma_H) + w(\gamma_X)
  \]
  is a coboundary in $\Omega^{2k}(A)$.
  To compute each part, we need to find an admissible transformation for each of the graphs $\gamma_I$, $\gamma_H$, and $\gamma_X$.
  Those transformations stem from choices \ref{item:admissible-perm-or-representative}--\ref{item:admissible-perm-lin-ord-flags-at-vertices}, which we make as follows.

  \begin{description}
    \item[Choice of \ref{item:admissible-perm-or-representative}:]
      Choose a linear orientation representative for $\gamma_I$ such that the first two vertices $v_1$ and $v_2$ are respectively the vertices at the top and the bottom of the I part, and such that the edge connecting them is oriented from top to bottom.
      Choose linear orientation representatives for $\gamma_H$ and $\gamma_X$ such that all parts common to all three diagrams
      (i.e., the vertices $v_3,\dots,v_{2k}$ and the orientations of all edges except the central edge in the disks) 
      agree in $\gamma_I$, $\gamma_H$ and $\gamma_X$.
      Then the remaining parts are fixed, up to a simultaneous swapping of $v_1$ and $v_2$ and flipping of the orientation of the central edge, and are given by
      \[
        \vcenter{\hbox{\includegraphics{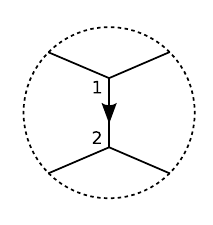}}} ,
        \qquad
        \vcenter{\hbox{\includegraphics{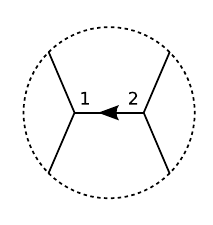}}} ,
        \qquad
        \vcenter{\hbox{\includegraphics{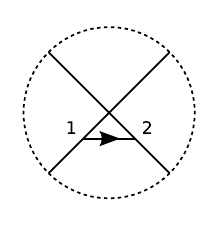}}} .
      \]

    \item[Choice of \ref{item:admissible-perm-lin-ord-edges}:]
      Put as first edge the central edge of each disk and have the same ordering for the remaining edges in all three diagrams.

    \item[Choice of \ref{item:admissible-perm-lin-ord-flags-at-vertices}:]
      Choose the orderings of the flags inside the disks as shown below.
      \[
        \vcenter{\hbox{\includegraphics{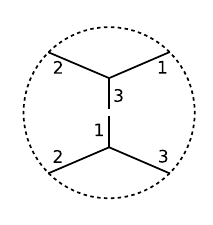}}} ,
        \qquad
        \vcenter{\hbox{\includegraphics{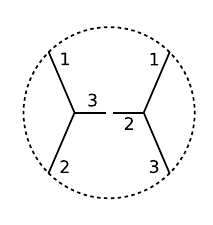}}} ,
        \qquad
        \vcenter{\hbox{\includegraphics{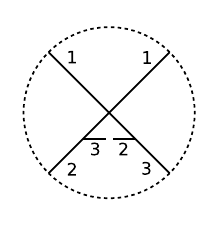}}} .
      \]
      Pick any ordering for the flags at the vertices outside the disks, with the constraint that these orderings coincide in all three diagrams $\gamma_I$, $\gamma_H$, and $\gamma_X$.
  \end{description}

  With these choices, 
  we can decompose the value of the weight system on $\gamma_I$ as a pairing between two contributions:
  \begin{dmath*}
    \left(\omega^{-1}\right)^{\otimes 3k} \circ \pi_I \circ \left(\tilde R^\nabla\right)^{\wedge 2k}
    = \langle \beta, \delta_I \rangle .
  \end{dmath*}
  Here,
  \begin{itemize}
    \item $\delta_I$ is an $((L/A)^*)^{\otimes 4}$-valued 2-form on $A$ representing the contribution of the vertices 1 and 2 and of the oriented central edge connecting them; 
    \item $\beta$ is an $(L/A)^{\otimes 4}$-valued $(6k-2)$-form on $A$ (actually, a $(6k-2)$-cocycle) representing the contribution of the rest of the diagram;
    \item the pairing $\left< \cdot,\cdot\right>$ is the wedge product on the $A$-forms and the pairing on their values.
  \end{itemize}
  Let us describe $\delta_I$ more precisely. We label the six flags of I as follows
  \[
    \vcenter{\hbox{\includegraphics{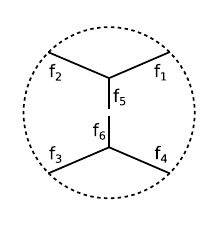}}} .
  \]
  Since $\delta_I$ represents the contribution of two vertices and one edge, it will be formed from 
  \begin{enumerate}[label=(\arabic*)]
    \item two copies of the Atiyah cocycle: $\tilde R^\nabla \wedge \tilde R^\nabla\in \Omega^2(A,((L/A)^*)^6)$, 
    \item one copy of $\omega^{-1}$, which we will choose to be in first position: $\omega^{-1}\otimes \id^{\otimes 4}:((L/A)^*)^6 \to ((L/A)^*)^4$, and
    \item a permutation $\pi$ of the six flags of diagram I sending the ordering given by \ref{item:admissible-perm-lin-ord-vert} and \ref{item:admissible-perm-lin-ord-flags-at-vertices} to the ordering given by \ref{item:admissible-perm-or-edges}, \ref{item:admissible-perm-lin-ord-edges}, and the additional requirement that ``the four legs of the I part oriented North-East, North-West, South-West, and South-East should be identified with the components 1, 2, 3, and 4, respectively, of $((L/A)^*)^{\otimes 4}$ in the value of $\delta_I$'' (thus after contracting with $\omega^{-1}$).
    
      The permutation $\pi$ thus sends the ordering $f_1, f_2, f_5, f_6, f_3, f_4$ to the ordering $f_5, f_6, f_1, f_2, f_3, f_4$.
      Hence, we have $\pi=(13)(24)$.
  \end{enumerate}
  We obtain
  \begin{dmath*}
    \delta_I 
    = \left( \omega^{-1} \otimes \id^{\otimes 4}\right)
    \circ
    \tau_{(13)(24)}
    \circ
    \left( \tilde R^\nabla \wedge \tilde R^\nabla \right)
    =
    \left( \omega^{-1} \otimes \id^{\otimes 4}\right)
    \circ
    \tau_{(13)(24)}
    \circ
    \left( {\id} \otimes {\id} \otimes {\omega^\flat} \otimes {\id} \otimes {\id} \otimes {\omega^\flat} \right)
    \circ
    \left( R^\nabla \wedge R^\nabla \right)
    =
    \left( \omega^{-1} \otimes \id^{\otimes 4}\right)
    \circ
    \left( {\omega^\flat} \otimes {\id} \otimes {\id} \otimes {\id} \otimes {\id} \otimes {\omega^\flat} \right)
    \circ
    \tau_{(13)(24)}
    \circ
    \left( R^\nabla \wedge R^\nabla \right)
    = \left( {\Tr} \otimes {\id^{\otimes 3}} \otimes {\omega^\flat} \right)
    \circ
    \tau_{(13)(24)}
    \circ
    \left( R^\nabla \wedge R^\nabla \right) .
  \end{dmath*}
  At the last line, $\Tr:L/A\otimes (L/A)^*\to M\times \fK$ is the canonical pairing and we used the identity
  \[
    {\Tr} 
    = \omega^{-1} \circ \left( \omega^\flat \otimes \id \right) : L/A \otimes (L/A)^* \to M\times \fK .
  \]
  Recall that $\omega^{-1}:(L/A)^*\wedge (L/A)^*\to M\times \fK$ is defined by $(\omega^{-1})^\sharp=(\omega^\flat)^{-1}$. Here $(\omega^{-1})^\sharp:(L/A)^*\to L/A:\alpha\mapsto \iota_\alpha\omega^{-1}$ and $\omega^\flat:L/A\to(L/A)^*:v\mapsto \iota_v\omega$.

  Now, we can check that
  \[
    \left\langle \delta_I , b_1\otimes b_2 \otimes b_3\otimes b_4 \right\rangle
    =
    \omega
    \left.\Big(
      \left\lceil R_2(R_2(b_1,b_2),b_3) \right\rfloor
      ,
      b_4
    \Big)\right.
  \]
  for all $b_1,b_2,b_3,b_4\in\Gamma(L/A)$, where the notation is taken from Theorem~\ref{thm:Corollary71-from-CSX}.

  Similarly, and with the same requirement about the ordering of the four legs of the diagrams, the value of the weight system on $\gamma_H$ and $\gamma_X$ can be computed by $\left< \beta, \delta_H\right>$ and $\left<\beta, \delta_X \right>$ where
  \begin{align*}
    \delta_H 
    &= \left( \omega^{-1} \otimes \id^{\otimes 4}\right)
    \circ
    \tau_{(15234)}
    \circ
    \left( \tilde R^\nabla \wedge \tilde R^\nabla \right)
    \\
    &= - \left( {\Tr} \otimes {\id^{\otimes 3}} \otimes {\omega^\flat} \right)
    \circ
    \tau_{(15234)}
    \circ
    \left( R^\nabla \wedge R^\nabla \right)
  \end{align*}
  and
  \begin{align*}
    \delta_X
    &= \left( \omega^{-1} \otimes \id^{\otimes 4}\right)
    \circ
    \tau_{(13)(25)}
    \circ
    \left( \tilde R^\nabla \wedge \tilde R^\nabla \right)
    \\
    &= \left( {\Tr} \otimes {\id^{\otimes 3}} \otimes {\omega^\flat} \right)
    \circ
    \tau_{(13)(25)}
    \circ
    \left( R^\nabla \wedge R^\nabla \right) .
  \end{align*}

  Assembling the three parts, we compute, for all $b_1,b_2,b_3,b_4\in\Gamma(L/A)$,
  \begin{dmath*}
    \left\langle \delta_I - \delta_H + \delta_X, b_1\otimes b_2 \otimes b_3\otimes b_4 \right\rangle
    =
    \omega
    \left.\Big(
              \left\lceil R_2(R_2(b_1,b_2),b_3) \right\rfloor
    - \left( -\left\lceil R_2(b_1,R_2(b_2,b_3)) \right\rfloor \right)
    +         \left\lceil R_2(b_2,R_2(b_1,b_3)) \right\rfloor
    ,
    b_4
    \Big)\right.
    =
    \left\langle \omega \circ (\psi \times \id), b_1\otimes b_2 \otimes b_3\otimes b_4 \right\rangle
    =
    \left\langle \partial^A \left( \omega \circ (\phi \otimes \id) \right), b_1\otimes b_2 \otimes b_3\otimes b_4 \right\rangle
  \end{dmath*}
  where $\psi$ and $\phi$ are defined in Theorem~\ref{thm:Corollary71-from-CSX}. 

  Since $\beta$ is a cocycle, this proves that 
  \begin{dmath*}
    w(\gamma_I) - w(\gamma_H) + w(\gamma_X) 
    = \left\langle \beta, \delta_I - \delta_H + \delta_X\right\rangle
    = \partial^A 
      \left\langle 
        \beta, 
        \left( \omega \circ (\phi \otimes \id) \right) 
      \right\rangle
  \end{dmath*}
  is a coboundary.
  Hence, the IHX relation is satisfied by $\ol w$.
\end{proof}


\subsection{Chord diagrams}
\label{ssec:chord-diagrams}

As in the classical case of Rozansky--Witten invariants, the construction above easily extends to weight systems on chord diagrams---yielding knot invariants---starting from a symplectic Lie pair $(L,A,\omega)$ and an $A$-module $E$ (or a collection $E_1,\dots,E_n$ thereof, to get link invariants).
We only sketch the proof here.

\begin{defn}
  A \emph{chord diagram} of order $n$ is an oriented circle together with $n$ chords (i.e., $n$ unordered pairs of points on the circle, all points being distinct) up to orientation preserving diffeomorphisms of the circle.
\end{defn}

Chords are represented by dashed lines, as in
\[
  \vcenter{\hbox{\includegraphics{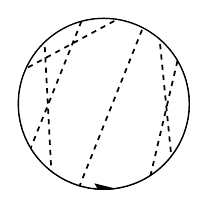}}}, 
  \qquad 
  \vcenter{\hbox{\includegraphics{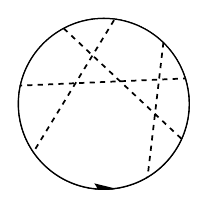}}} .
\]
Note that the first diagram has an \emph{isolated} chord, i.e., a chord not crossed by any other.
We will denote $\mathcal{C}$ (respectively, $\mathcal{C}_n$) the linear span (over $\fK$) of all chord diagrams (respectively, of order $n$).

\begin{defn}
  Let $V$ be a vector space over $\fK$. A $V$-valued \emph{weight system} of order $n$ is a linear map $w_n:\mathcal{C}_n\to V$ satisfying
  \begin{enumerate}[label=(\arabic*)]
    \item the 1T relation:
      $w_n$ vanishes on diagrams having an isolated chord;
    \item the 4T relation:
      \[
        w_n\left( \vcenter{\hbox{\includegraphics{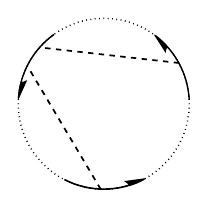}}} 
        - \vcenter{\hbox{\includegraphics{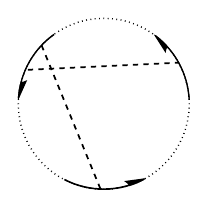}}} 
        - \vcenter{\hbox{\includegraphics{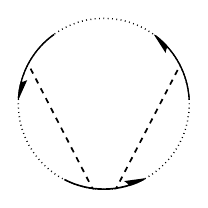}}} 
        + \vcenter{\hbox{\includegraphics{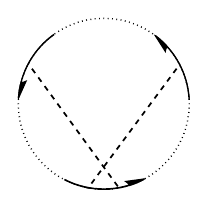}}} \right) = 0.
      \]
  \end{enumerate}
\end{defn}

As before, let $\nabla$ be an $L$-connection on $E$ extending the $A$-action. Its curvature defines an $A$-cocycle $R^\nabla\in\Gamma(A^*\otimes (L/A)^* \otimes \End(E))$ representing the Atiyah class $\alpha_E$ of $E$.
Let $D$ be a chord diagram with $k$ chords and choose an origin (not at a vertex) on the oriented circle so that the vertices are ordered $v_1,\dots,v_{2k}$.

Associate to each vertex a copy of $R^\nabla$ and take their tensor product, i.e., consider the element $R^\nabla \otimes \cdots \otimes R^\nabla$ in $\Gamma\left( \left(A^*\otimes (L/A)^* \otimes \End(E)\right)^{\otimes 2k}\right)$.
Compose the $\End(E)$ parts from left to right and then take the trace of the result.
Contract the $(L/A)^*$ parts with $\omega^{-1}$ along the chords, using the ordering coming from the choice of origin.
The result is an element of $\Gamma((A^*)^{\otimes 2k})$ which we project to the exterior product $\Gamma(\Lambda^{2k} A^*)$.

The proof that this gives a weight system goes along the same lines as that of Theorem~\ref{thm:T-weight-system} and uses two main ingredients, in addition to the total symmetry of the Atiyah class of $L/A$.
The first ingredient is (the cocycle version of) the fact that the Atiyah class of an $A$-module $E$ makes $E[-1]$ a module over the Lie algebra object $L/A[-1]$ in the derived category of $U(A)$-modules \cite[Sec.\ 2.5.5 and Thm. 45]{chen_atiyah_2012}.
Here, $U(A)$ is the universal enveloping algebra of the Lie algebroid $A$.
This is used to prove the 4T relation.
To get the 1T relation as well, one needs to project this almost-weight system to the space of weight systems. This only uses the Hopf algebra structure of the space $\mathcal{C}/\Span\left<4T\right>$ of chord diagrams modulo the 4T relations, and is thus identical to the metric Lie algebra case \cite[Sec.\ 3.2.2 and Ex.\ 3.16]{bar-natan_vassiliev_1995}.


\let\emph\oldemph


\bibliographystyle{abbrv}
\bibliography{Paper-RW-SymplecticLiePairs}

\end{document}